\documentclass[10pt,reqno]{amsart}
\usepackage{amssymb,mathrsfs,graphicx,extpfeil,amsmath}
\usepackage{ifthen,latexsym,float,colortbl}
\usepackage[margin=1in]{geometry}
\usepackage{caption}
\usepackage{sidecap}
\usepackage{rotating,dsfont,stackengine}
\usepackage{colortbl}
\definecolor{black}{rgb}{0.0, 0.0, 0.0}
\definecolor{red}{rgb}{1.0, 0.5, 0.5}
\provideboolean{shownotes} 
\setboolean{shownotes}{true} 
\newcommand{\margnote}[1]{
\ifthenelse{\boolean{shownotes}}%
{\marginpar{\raggedright\tiny\texttt{#1}}}%
{}%
}
\newcommand{\hole}[1]{
\ifthenelse{\boolean{shownotes}}%
{\begin{center} \fbox{ \rule {.25cm}{0cm} \rule[-.1cm]{0cm}{.4cm}
\parbox{.85\textwidth}{\begin{center} \texttt{#1}\end{center}} \rule
{.25cm}{0cm}}\end{center}} {} }

\graphicspath{{pics/}}

\title[The pressureless damped Euler--Riesz equations]{The pressureless damped Euler--Riesz equations}

\author[Choi]{Young-Pil Choi}
\address[Young-Pil Choi]{\newline Department of Mathematics \newline
Yonsei University, 50 Yonsei-Ro, Seodaemun-Gu, Seoul 03722, Republic of Korea}
\email{ypchoi@yonsei.ac.kr}

\author[Jung]{Jinwook Jung}
\address[Jinwook Jung]{\newline Research Institute of Basic Sciences \newline Seoul National University, Seoul  08826, Republic of Korea}
\email{warp100@snu.ac.kr}

\numberwithin{equation}{section}

\newtheorem{theorem}{Theorem}[section]
\newtheorem{lemma}{Lemma}[section]

\newtheorem{proposition}{Proposition}[section]
\newtheorem{remark}{Remark}[section]

\newcommand{\R}{\mathbb R}

\newcommand{\N}{\mathbb N}
\newcommand{\Z}{\mathbb Z}
\newcommand{\om}{\Omega}
\newcommand{\ls}{\lesssim}

\newcommand{\T}{\mathbb T}

\newcommand{\mc}{\mathcal C}

\newcommand{\bq}{\begin{equation}}
\newcommand{\eq}{\end{equation}}
\newcommand{\e}{\varepsilon}
\newcommand{\lt}{\left}
\newcommand{\rt}{\right}

\newcommand{\pa}{\partial}

\newcommand{\me}{\mathcal{E}}

\newcommand{\md}{\mathcal{D}}

\makeatletter
\def\moverlay{\mathpalette\mov@rlay}
\def\mov@rlay#1#2{\leavevmode\vtop{%
   \baselineskip\z@skip \lineskiplimit-\maxdimen
   \ialign{\hfil$\m@th#1##$\hfil\cr#2\crcr}}}
\newcommand{\charfusion}[3][\mathord]{
    #1{\ifx#1\mathop\vphantom{#2}\fi
        \mathpalette\mov@rlay{#2\cr#3}
      }
    \ifx#1\mathop\expandafter\displaylimits\fi}
\makeatother

\begin{document}
\allowdisplaybreaks

\date{\today}

\subjclass[]{}

\begin{abstract} In this paper, we analyze the pressureless damped Euler--Riesz equations posed in either $\R^d$ or $\T^d$. We construct the global-in-time existence and uniqueness of classical solutions for the system around a constant background state. We also establish large-time behaviors of classical solutions showing the solutions towards the equilibrium as time goes to infinity. For the whole space case, we first show the algebraic decay rate of solutions under additional assumptions on the initial data compared to the existence theory. We then refine the argument to have the exponential decay rate of convergence even in the whole space. In the case of the periodic domain, without any further regularity assumptions on the initial data, we provide the exponential convergence of solutions.
\end{abstract}

\maketitle \centerline{\date}


%
%
%
%
\section{Introduction}
In this paper, we are interested in the global well-posedness and large-time behavior for the pressurelss Euler--Riesz equations with linear damping posed either in $\om =\R^d$ or $\T^d$:
\begin{align}\label{np_ER}
\begin{aligned}
&\pa_t \rho + \nabla \cdot (\rho u) = 0, \quad x \in \Omega, \quad t > 0,\cr
&\pa_t (\rho u) + \nabla \cdot (\rho u \otimes u) = -\gamma \rho u -\lambda \rho \nabla \Lambda^{\alpha-d} (\rho-c),
\end{aligned}
\end{align}
subject to the initial data
\bq\label{ini_ER}
(\rho,u)|_{t = 0} =: (\rho_0, u_0), \quad x \in \Omega,
\eq
where $\rho=\rho(x,t)$ and $u= u(x,t)$ denote the density and velocity  of the fluid at time $t$ and position $x$, respectively. Here, the Riesz operator $\Lambda^{s}$ is defined by $(-\Delta)^{s/2}$, and we concentrate on the case $d-2<\alpha<d$. The coefficients $\gamma$ and $\lambda$ are positive constants, and $c>0$ is the nonzero background state. For $\Lambda^{\alpha-d} (\rho-c)$ to be well-defined, we impose the neutrality condition: 
\[
\int_\Omega (\rho - c)\,dx = 0.
\]
Without loss of generality, for simplicity of presentation, we set $\gamma = \lambda = c=1$.

The pressureless Euler--Riesz system is recently derived in \cite{Ser20} from the $N$-interacting particle system governed by Newton's laws. In \cite{Ser20}, the interaction between particles is given by the force fields $\nabla K$ where $K(x) = |x|^{-\alpha}$ with $d-2 < \alpha < d$, and the modulated kinetic and interaction energies are employed to show the quantitative error estimate between the particle and Euler--Riesz systems. We would like to remark that the case $\alpha=d-2$ with $d \geq 3$ corresponds to the Coulomb potential, and the case $\max\{d-2,0\} < \alpha < d$ with $d \geq 1$ is called Riesz potential.
The local well-posedness theory for the system \eqref{np_ER} is developed in \cite{CJpre}. Strictly speaking, in \cite{CJpre}, the system \eqref{np_ER} with zero background state in the undamped case, i.e. $c=0$ and $\gamma=0$, is considered, however, a small modification of the strategy used in \cite{CJpre} leads to establishing the local existence and uniqueness of classical solutions to the system \eqref{np_ER}, see Theorem \ref{thm_local} below for more detailed discussion.

The main purpose of the current work is to establish the global-in-time existence and uniqueness of classical solutions to the pressureless damped Euler--Riesz system \eqref{np_ER} and its large-time behavior. We would like to emphasize that to the best of our knowledge, even for the multi-dimensional Euler--Poisson system, in the absence of the pressure, much less is known about the global-in-time regularity of classical solutions or the large-time behavior estimate. For the one-dimensional case, a critical threshold on the initial data distinguishing the global-in-time regularity of solutions and finite-time singularity formation for the pressureless Euler--Poisson system is analyzed in \cite{BL20, CCZ16, EHT01}, see also \cite{CCTT16, TW08} for the case with pressure and other related systems. For higher-dimensional problems, the critical threshold estimate for the 2D {\it restricted} Euler--Poisson system is studied in \cite{TL03}, see \cite{LT02} for more general discussion on the restricted flows. The global existence of smooth solutions for the Euler--Poisson system around a constant background state is discussed in \cite{IP13, JLZ14, LW14}. We also refer to \cite{G98, GP11, HJ19} for the three-dimensional problems.

In order to state our first main result concerns the global well-posedness theory, we use $\rho>0$ and let $h= \rho-1$ to reformulate the system \eqref{np_ER}-\eqref{ini_ER} as
\bq\label{np_ER2}
\begin{aligned}
&\pa_t h + \nabla\cdot (hu) + \nabla \cdot u = 0, \quad x \in \Omega, \quad t > 0,\cr
&\pa_t u + u\cdot \nabla u = -u -\nabla \Lambda^{\alpha-d} h,
\end{aligned}
\eq
with initial data
\bq\label{ini_ER2}
(h,u)|_{t=0} =: (h_0 := \rho_0-1,u_0), \quad x \in \Omega.
\eq
For our solution spaces, we consider the following norms:
\bq\label{def_Xm}
\|(h,u)\|_{X^m}^2 := \|h\|_{H^{m}}^2 + \|u\|_{H^{m+\frac{d-\alpha}{2}}}^2 + \|h\|_{\dot{H}^{-\frac{d-\alpha}{2}}}^2.
\eq 
The notation $X^m$ naturally denotes the space of functions with finite corresponding norm.
 
Now, we state our first result on the global existence and uniqueness of solutions to \eqref{np_ER}.

\begin{theorem}\label{main_thm}
Consider the system \eqref{np_ER2} on either $\om = \R^d$ or $\T^d$ with $d \ge 1$ and  $\max\{d-2,0\} < \alpha <d$. For any $m>\frac{d}{2} + 2$,  suppose that the initial data $(\rho_0, u_0)$ satisfy 
\[
\inf_{x\in\om} h_0(x) >-1, \quad \int_\Omega h_0(x) \,dx = 0, \quad \mbox{and} \quad (h_0, u_0) \in X^m. 
\] 
If
\[
\|(h_0, u_0)\|_{X^m}< \e_1,
\]
for some $\e_1>0$ sufficiently small, the the system \eqref{np_ER2}-\eqref{ini_ER2} admits a unique solution in $\mc(\R_+; X^m)$. 
\end{theorem}

As mentioned above, the local-in-time existence of solutions for the system \eqref{np_ER} with $\gamma=0$ is investigated in a recent work \cite{CJpre}. We are currently interested in the linear velocity-damping effect on the global-in-time regularity of solutions, and as stated in Theorem \ref{main_thm}, the damping can prevent the finite-time breakdown of smoothness of solutions, even in the absence of pressure, when the initial data are sufficiently small and regular. This is reminiscent of the proof of a global Cauchy problem for the compressible Euler equations with damping \cite{STW03,TW13}. However, we only have the Riesz interactions, not the pressure term. On the other hand, for the Euler--Poisson system around a constant background state, i.e. the system \eqref{np_ER} with pressure, $\alpha=d-2$, and $\gamma=0$, two-and three-dimensional Cauchy problems are first discussed in \cite{IP13, LW14} and in \cite{G98} under the consideration of irrotational flows. Compared to those works, we have the linear damping in velocity instead of the pressure term. The main difficulties lie in the analysis of the highest-order derivative estimates and the dissipation rate on the solutions due to the singularity of Riesz interactions, beyond the Coulomb ones. It is natural to expect that the linear damping gives a good dissipation rate for the velocity $u$. However, it is not clear how to analyze the stabilizing effect from the Riesz interactions and obtain a proper dissipation rate for the perturbed density $h$. In order to overcome those difficulties, inspired by \cite{CJpre}, we estimate our solutions in the fractional Sobolev space specified in \eqref{def_Xm} and a modified $H^m$ norm for $h$, see \eqref{def_tH} below, to have some cancellations of terms with the highest-order derivatives. For the dissipation estimates for $h$, we clarify the dispersive effect of the Riesz interaction and establish a delicate {\it hypocoercivity}-type estimate which provides the higher-order dissipation rate. The proof highly relies on the energy method based on the commutator estimates for the fractional Laplacian and Gagliardo--Nirenberg--Sobolev-type inequalities. 

\begin{remark}In \cite{CJpre3}, after a suitable scaling, the strong relaxation limit of the system \eqref{np_ER} with the zero background state, i.e. $c=0$, is investigated, and the following fractional porous medium equation \cite{CSV13,CV11}  is rigorously and quantitatively derived:
\[
\pa_t \rho = \nabla \cdot (\rho\nabla \Lambda^{\alpha-d} \rho).
\]
We would like to remark that the argument used in \cite{CJpre3} can be extended to the nonzero constant background state case when $\om = \T^d$. The local-in-time existence and uniqueness of classical solutions for that limiting equation are recently established in \cite{CJpre2}. Note that as long as there exist classical solutions for those systems, the strong relaxation limit estimate holds. Thus as a by-product of Theorem \ref{main_thm} if one can show the existence of global-in-time classical solutions to the porous medium equation, then the relaxation limit holds for all times.
We also refer to \cite{CCT19, CPW20, Cpre, HMP05, LT17} for the strong relaxation limits of compressible Euler/Euler--Poisson systems. 
\end{remark}

Our second result provides the large-time behavior of solutions, obtained in Theorem \ref{main_thm}, to system \eqref{np_ER} showing the algebraic or exponential decay rates of convergence of solutions in $X^m$ when $\Omega = \R^d$ or $\T^d$.
\begin{theorem}\label{main_thm2} Let $d \geq 2$ and  the assumptions of Theorem \ref{main_thm} be satisfied. 
\begin{enumerate}
\item[(i)] (Whole space case): 
If we additionally assume that  
\bq\label{as_lt}
(h_0, u_0) \in \dot{H}^{-s-\frac{d-\alpha}{2}}(\R^d) \times [\dot{H}^{-s}(\R^d)]^d
\eq
for some 
\bq\label{as_lt2}
s \in \lt[1 - \frac{d-\alpha}{2}, \,\frac\alpha2 \rt],
\eq
then we have
\[
\|(h,u)(\cdot,t)\|_{X^m}^2 + \|h\|_{\dot{H}^{-1+\frac{d-\alpha}{2}}}^2 + \|u\|_{\dot{H}^{d-\alpha-1}}^2 \le C(1+t)^{-\eta}, \quad t \ge 0, \quad .
\]
where $\eta > 0$ is given by
\[
\eta := \min\lt\{\frac{2s}{d-\alpha}, \frac{s+d-\alpha-1}{1-\frac{d-\alpha}{2}}\rt\}.
\]
Furthermore, if the order $s > 0$ is large enough such that $(\frac\alpha2 \geq )s>2+d-\alpha$, then we have the exponential decay rate of convergence:
\[
\|(h,u)(\cdot,t)\|_{X^m}^2 \le Ce^{-\zeta t}, \quad  t \ge 0,
\]
for some positive constants $C$ and $\zeta$ independent of $t$.

\item[(ii)] (Periodic case): There exist positive constants $C$ and $\lambda$ independent of $t$ such that 
\[
\|(h,u)(\cdot,t)\|_{X^m}  \le Ce^{-\lambda t}, \quad t \ge 0.
\]
\end{enumerate}
\end{theorem}

\begin{remark} The condition \eqref{as_lt2} naturally requires  $d \ge 2$. 
\end{remark}

\begin{remark}\label{rmk_decay} The assumption \eqref{as_lt2} and the dimension restriction $d \geq 2$ can be relaxed in the whole space case.  Indeed, if we only assume \eqref{as_lt} for some $s \in (0, \alpha/2]$, then we have
\bq\label{decay_weak}
\|(h,u)(\cdot,t)\|_{X^m}^2  \le C(1+t)^{-\frac{s}{1+\frac{d-\alpha}{2}}}, \quad t \ge 0,
\eq
where $C$ is a positive constant independent of $t$. Note that when $s+\frac{d-\alpha}{2} <1$, the decay rate of convergence for the whole space case is at most $(1+t)^{-\frac{1-\frac{d-\alpha}{2}}{1+\frac{d-\alpha}{2}}} = (1+t)^{-1+\e}$ for some constant $\e \in (0,1)$. In this case, even though the order $s > 0$ is only assumed to be positive, the decay rate does not depend on the dimension $d$, however this decay estimate provide a good decay estimate for the one dimensional case. On the other hand, when $s=\frac\alpha2$, i.e. \eqref{as_lt2} holds, the decay rate becomes $(1+t)^{-\min\lt\{\frac{\alpha}{d-\alpha}, \frac{d-\frac\alpha2-1}{1-\frac{d-\alpha}{2}}\rt\}}$ and it becomes $(1+t)^{-(d-1)}$ if $\alpha=d-1$. This shows that we have a better decay rate of convergence in higher dimensions.
\end{remark}

\begin{remark}\label{rmk_lt} For the periodic domain case, if we are only interested in the large-time behavior of the lowest order norm, i.e. $\|u\|_{L^2}+\|h\|_{\dot{H}^{-\frac{d-\alpha}{2}}}$, then the smallness assumption on the solutions is not necessarily required. More precisely, if we assume
\begin{enumerate}
\item[(i)]
$\inf\limits_{(x,t)\in \T^d \times \R_+} 1 + h(x,t) \ge h_{min}>0$,
\item[(ii)]
$h \in W^{1,\infty}(\T^d \times \R_+)$, $\nabla u \in L^\infty(\R_+;[L^{\infty}(\T^d)]^d)$,
\end{enumerate}
then we have
\[
\|u(\cdot,t)\|_{L^2}+\|h(\cdot,t)\|_{\dot{H}^{-\frac{d-\alpha}{2}}} \le Ce^{-\lambda t}.
\]
Here $C$ and $\lambda$ are positive constants independent of $t$. In fact, the above estimate plays a crucial role in establishing the exponential decay rate of convergence of $\|(h,u)(\cdot,t)\|_{X^m}^2$, see Proposition \ref{main_prop} below. 
\end{remark}

\begin{remark}All the results in Theorems \ref{main_thm} and \ref{main_thm2} can be readily extended to the Coulomb interaction case, i.e. the system \eqref{np_ER2} with $\alpha = d-2$. In particular, if $d > 6$ and \eqref{as_lt} holds with $s \in (4,d-2)$, we have the exponential decay rate of convergence of solutions for the system \eqref{np_ER2} with $\alpha = d-2$, i.e. pressureless damped Euler--Poisson system, even in the whole space. 
\end{remark}

For the whole space case, as stated in Theorem \ref{main_thm2}, we take into account the negative Sobolev space of solutions. The negative Sobolev norm is first used in \cite{GW12} for the estimates on the optimal time decay rates of convergence of solutions to the dissipative equations in the whole space. As mentioned above, we were able to show that the hypocoercivity-type estimate produces the dissipation rate for $h$, however, it does not give the lower-order norm for $h$. For this, we find a proper negative order of derivative of solutions that closes the estimates of Sobolev negative norms, and thus the algebraic decay rate of convergence of solutions is established. We would also like to emphasize that the exponential decay rate is found when the negative order is sufficiently large, which subsequently requires $d \geq 1$ large enough, in the whole space. On the other hand, for the periodic domain case, we suitably use  the monotonicity of the negative Sobolev norms and construct a modulated energy for the system \eqref{np_ER2}. More precisely, the modulated energy is equivalent to the lowest order norm of solutions, $\|u\|_{L^2}+\|h\|_{\dot{H}^{-\frac{d-\alpha}{2}}}$, and this decays to zero exponentially fast as time goes to infinity. This strategy does not require any further integrability in the negative Sobolev space and as stated in Remark \ref{rmk_lt} any smallness assumption on solutions is not needed. This decay estimate on lower-order norm of solutions together with the energy estimate established in the proof of Theorem \ref{main_thm} yields the exponential decay rate of convergence of solutions in $X^m$ norm.

Throughout this paper, we denote by $C$ a generic positive constant which may differ from line to line, $C = C(\alpha,\beta,\dots)$ represents the positive constant depending on $\alpha,\beta,\dots$. $f \ls g$ and $f \approx g$ mean that there exists a positive constant $C>0$ such that $f \leq Cg$ and $C^{-1} f \leq  g \leq Cf$, respectively. $f \ls_{\alpha,\beta,\dots} g$ denotes $f \leq C(\alpha,\beta,\dots) g$ for some constant $C(\alpha,\beta,\dots) > 0$. $\pa^k$ denotes any partial derivative of order $k$.

The rest of this paper is organized as follows. In Section \ref{sec_pre}, we introduce several auxiliary lemmas regarding the commutator estimates and Sobolev embeddings. These estimates will very often be used throughout the paper. Section \ref{sec_ext} is devoted to provide the details on the proof of our first main theorem, Theorem \ref{main_thm}. Since the local well-posedness is by now classical, we mainly discuss the {\it a priori} estimates of solutions in the proposed solution space. This yields that the local-in-time solutions can be extended to the global-in-time one. Finally, in Section \ref{sec_lt}, we study the large-time behavior of classical solutions. 

%
%
%
%

\section{Preliminaries}\label{sec_pre}
In this section, we provide various technical lemmas that will be significantly used  throughout the paper.

We first recall from \cite{CJpre} the commutator estimate.
\begin{lemma}\cite{CJpre}\label{comm_est2}
Let $s \ge 0$. For a vector field $v \in (H^{\frac{d}{2} + 1 + s + \e}(\R^d))^d$ and $f\in H^s(\R^d)$, we have 
\[
\|[\Lambda^s, v \cdot \nabla]f\|_{L^2} \lesssim_{s,d,\e} \|v\|_{H^{\frac{d}{2} + 1 + s +\e}} \|f\|_{H^s}.
\]
\end{lemma}
We next present several results on the Gagliardo--Nirenberg interpolation inequalities and Moser-type inequalities. 
\begin{proposition}\label{ineqs}
We have the following relations:
\begin{enumerate}
\item[(i)]
If $f\in H^{s_2}(\om)$ and $0\le s_1 \le s_2$,
\[\|f\|_{H^{s_1}} \le \|f\|_{L^2}^{\frac{s_2-s_1}{s_2}} \|f\|_{H^{s_2}}^{\frac{s_1}{s_2}}.\]
\item[(ii)]
If $s_2, s_3 \ge 0$ and $0 \le s_1 \le \min\{s_2, s_3, s_2+s_3-d/2\}$,
\[
\|fg\|_{H^{s_1}} \lesssim_{d,s_1, s_2, s_3} \|f\|_{H^{s_2}} \|g\|_{H^{s_3}}.
\]
\item[(iii)]
If $ j, \ell \in \N$ with $0 \le j \le \ell$ and $f \in H^\ell(\om)$, 
\[\|\nabla^j f\|_{L^2} \lesssim_{d,j,\ell} \|\nabla^\ell f\|_{L^2}^{\frac{j}{\ell}}\|f\|_{L^2}^{1-\frac{j}{\ell}}.\] 
\item[(iv)]
(Moser-type inequality)
If $f,g \in (H^k \cap L^\infty)(\om)$,
\[\|\pa^k(fg)\|_{L^2} \lesssim_{d,k} \|f\|_{L^\infty} \|\pa^k g\|_{L^2} + \|g\|_{L^\infty}\|\pa^k f\|_{L^2}.\]
Moreover, if $\nabla g \in L^\infty(\om)$,
\[\|\pa^k(fg) - (\pa^k f) g\|_{L^2} \lesssim_{d,k} \|f\|_{L^\infty} \|\pa^k g\|_{L^2} + \|\nabla g\|_{L^\infty}\|\pa^{k-1} f\|_{L^2}.\]
In addition, if $f,g \in H^k (\om)$ with $\nabla f, \nabla g \in L^\infty(\om)$,
\[\|\pa^k(fg) - (\pa^k f) g - f (\pa^k g)\|_{L^2} \lesssim_{d,k} \|\nabla f\|_{L^\infty} \|\pa^{k-1} g\|_{L^2} + \|\nabla g\|_{L^\infty}\|\pa^{k-1} f\|_{L^2}.\]
\end{enumerate}
\end{proposition}

We finally show the total energy estimate of the system \eqref{np_ER} whose proof can be readily obtained.
\begin{proposition}\label{zeroth}
For $T>0$, let $(\rho, u)$ be a classical solution to \eqref{np_ER} on $[0,T]$. Then we have
\[
\frac12 \frac{d}{dt}\lt( \int_\om \rho |u|^2\,dx + \int_\om (\rho-1)\Lambda^{\alpha -d}(\rho-1)\,dx \rt) + \int_\om \rho |u|^2\,dx = 0.
\]
\end{proposition}

%
%
%
%

\section{Global well-posedness for the damped pressureless Euler--Riesz system}\label{sec_ext}
In this section, we present the proof of Theorem \ref{main_thm}.  Although our proof mostly considers the case $\om = \R^d$, similar arguments can be used for the case $\om =\T^d$. 

\subsection{Local well-posedness}

Note that the local-in-time existence and uniqueness of strong solutions can be deduced from \cite[Theorem 3.1]{CJpre}. Strictly speaking in \cite{CJpre}, the local well-posedness theory is studied in the case that $\rho$ is integrable in $\Omega$, however the proof can be readily extended to our case. Thus we present the following theorem without providing any details on its proof.

\begin{theorem}\label{thm_local} Let the same assumptions as in Theorem \ref{main_thm} be verified. Then for any positive constants $\e_1 < M_0$, there exists a positive constant $T_0$ depending only on $\e_1$ and $M_0$ such that if $\|(h_0, u_0)\|_{X^m}< \e_1$, then the system \eqref{np_ER2} admits a unique solution $(h,u) \in \mc([0,T); X^m)$ satisfying
\[
\sup_{0 \leq t \leq T_0} \|(h,u)\|_{X^m} \leq M_0.
\]
\end{theorem}

We next show the equivalence relation between the reformulated system \eqref{np_ER2} and the original one \eqref{np_ER}. Since its proof is classical, we omit it here.

\begin{proposition} Let $m > \frac d2+2$. For any fixed $T>0$, if $(\rho,u) \in \mc([0,T); X^m)$ solves the system \eqref{np_ER} with $\rho > 0$, then $(h,u)  \in \mc([0,T); X^m)$ solves the system \eqref{np_ER2} with $1+h > 0$. Conversely, if $(h,u)  \in \mc([0,T); X^m)$ solves the system \eqref{np_ER2} with $1+h > 0$ solves the system \eqref{np_ER2} with $1+h > 0$, then $(\rho,u) \in \mc([0,T); X^m)$ solves the system \eqref{np_ER} with $\rho > 0$.
\end{proposition}

\subsection{Global well-posedness}
In this part, we focus on the {\it a priori} estimates of solutions $(h,u)$ in the function space $L^\infty(0,T;X^m)$.  

Before we move on, we define a modified $H^m$ norm for $h$ as follows:
\bq\label{def_tH}
\|h\|_{\tilde{H}^m} := \sum_{0<|k|\le m} \|\rho^{-\frac12} \pa^k h\|_{L^2}.
\eq
Note that $\|h\|_{\tilde{H}^{m}} = \|\rho\|_{\tilde{H}^m}$. Furthermore, if $\|h\|_{L^\infty} <1$, then
\[
\|h\|_{H^m} \approx \|h\|_{L^2} + \|h\|_{\tilde{H}^m},
\]
since
\[
\lt( 1-\|h\|_{L^\infty}\rt)^{1/2} \|h\|_{\tilde{H}^m} \leq \sum_{0<|k|\le m} \| \pa^k h\|_{L^2} \le \|\rho\|_{L^\infty}^{\frac12} \|h\|_{\tilde{H}^m}.
\]
Thus, rather than directly estimate $\|h\|_{H^m}$, we can estimate $\|h\|_{L^2} + \|h\|_{\tilde{H}^m}$. 
 
\vspace{0.4cm}

Next, we investigate higher-order estimates for $(h, u)$. Before proceeding, for notational simplicity, we set
\[
X(T;m) := \sup_{0 \leq t \leq T}\|(h,u)(\cdot,t)\|_{X^m}^2 \quad \mbox{and} \quad X_0(m) := \|(h_0,u_0)\|_{X^m}^2.
\] 
Since the proofs for the following two lemmas are almost the same as in \cite{CJpre}, we omit here.

\begin{lemma}\label{est_u}
Let $T>0$, $m > \frac d2+2$, and $(h,u) \in \mc([0,T); X^m)$  be a solution to the system \eqref{np_ER2}. Then we have
\[
\frac12 \frac{d}{dt}\|U_k\|_{L^2}^2 + \|U_k\|_{L^2}^2 \le C\|u\|_{H^{m+\frac{d-\alpha}{2}}}^3 - \int_\om \Lambda^{\frac{\alpha-d}{2}}\nabla R_k \cdot U_k\,dx,
\]
for $0 \le k \le m$, where $U_k := \Lambda^{\frac{d-\alpha}{2}} \pa^k u$, $R_k := \pa^k h$ and $C=C(m,k,d,\alpha)$ is a positive constant independent of $T$.
\end{lemma}

\begin{remark}
Thanks to the Gagliardo--Nirenberg interpolation inequalities in Proposition \ref{ineqs}, we have the following equivalence relation: for any $i \in \{0, \dots, m\}$,
\bq\label{equiv_u}
\|u\|_{L^2} + \sum_{i \le k \le m} \|U_k\|_{L^2} \approx \|u\|_{H^{m+\frac{d-\alpha}{2}}}.
\eq
\end{remark}

\begin{lemma}\label{est_h}
Let $T>0$, $m > \frac d2+2$,  and $(h,u) \in \mc([0,T); X^m)$  be a solution to the system \eqref{np_ER2}. Then we have
\[
\frac12\frac{d}{dt}\int_\om \frac1\rho R_k^2\,dx \le C\|u\|_{H^{m+\frac{d-\alpha}{2}}} \lt( 1+\|\nabla \log \rho\|_{L^\infty}\rt)^{2(k-1)}\sum_{0<l\le k} \lt\| \frac{1}{\sqrt{\rho}}R_l\rt\|_{L^2}^2 +\int_\om \Lambda^{\frac{\alpha-d}{2}}\nabla R_k \cdot U_k\,dx,
\]
for $1\le k \le m$, where $C=C(m,k,d,\alpha)$ is a positive constant independent of $T$.
\end{lemma}
Here, we separately consider the $L^2$-estimate for $h$.
\begin{lemma}\label{est_h_0}
Let $T>0$ and $(h,u) \in \mc([0,T); X^m)$  be a solution to the system \eqref{np_ER2}. Then we have
\[
\frac12\frac{d}{dt}\|h\|_{L^2}^2 \le  \|u\|_{L^2}\|\nabla h\|_{L^2}\|h\|_{L^\infty} + \int_\om \Lambda^{\frac{\alpha-d}{2}}\nabla R_0 \cdot U_0\,dx.
\]
\end{lemma}
\begin{proof}
Direct computation implies
\begin{align*}
\frac12\frac{d}{dt}\|h\|_{L^2}^2&= -\int_\om h \nabla \cdot (hu)\,dx - \int_\om h \nabla \cdot u\,dx\\
&= \int_\om h (\nabla h \cdot u)\,dx + \int_\om \nabla h \cdot u\,dx\\
&\le \|h\|_{L^\infty} \|\nabla h\|_{L^2}\|u\|_{L^2} + \int_\om \Lambda^{\frac{\alpha-d}{2}}\nabla R_0 \cdot U_0\,dx,
\end{align*}
and this implies the desired result.
\end{proof}

As stated in Lemma \ref{est_u}, due to the presence of the linear damping in velocity, we have a dissipation rate for the velocity $u$. Moreover, the terms with the highest order derivatives appeared in Lemmas \ref{est_u} and \ref{est_h} are canceled each other out. Thus we now focus on the estimate for the dissipation rate for $h$. For this, a delicate analysis for the Riesz interaction term based on the hypocoercivity-type estimate is required. 

We first begin with the zeroth-order estimate.
\begin{lemma}\label{h_decay0}
Let $T>0$, $m > \frac d2+2$, and $(h,u) \in \mc([0,T); X^m)$  be a solution to the system \eqref{np_ER2} satisfying
\[
\sup_{0\le t \le T} \|h(t)\|_{L^\infty} \le \frac12.
\]
 Then we have
 \[
\frac{d}{dt}\int_\om \frac1\rho  \nabla h \cdot \Lambda^{d-\alpha} u\,dx + \frac12\|\nabla h\|_{L^2}^2 \le  C\|\nabla h\|_{H^{m-1}} \|u\|_{H^{m+\frac{d-\alpha}{2}}}^2 + \|\Lambda^{\frac{d-\alpha}{2}} (\nabla \cdot u)\|_{L^2}^2 + 6\|\Lambda^{d-\alpha} u\|_{L^2}^2,
\]
where $C=C(m,d,\alpha)$ is a positive constant independent of $T$.
\end{lemma}
\begin{proof} We first find
\[\begin{aligned}
\frac{d}{dt}&\int_\om \frac1\rho  \nabla h \cdot \Lambda^{d-\alpha} u\,dx\\
&=-\int_\om \frac{\pa_t \rho}{\rho^2} \nabla h \cdot \Lambda^{d-\alpha} u\,dx + \int_\om \frac1\rho \nabla (\pa_t h) \cdot \Lambda^{d-\alpha} u\,dx + \int_\om \frac1\rho \nabla h \cdot \Lambda^{d-\alpha} (\pa_t u)\,dx\\
&=: \mathcal{I}_1 + \mathcal{I}_2 + \mathcal{I}_3,
\end{aligned}\]
where we use the equation of $h$ in \eqref{np_ER2} to estimate
\begin{align*}
\mathcal{I}_2 &= \int_\om \frac{\nabla h}{\rho^2} \pa_t h \cdot \Lambda^{d-\alpha} u\,dx - \int_\om \frac1\rho \pa_t h \Lambda^{d-\alpha} (\nabla \cdot u)\,dx\\
&= -\mathcal{I}_1 + \int_\om \frac{1}{\rho}\nabla \cdot (\rho u) \Lambda^{d-\alpha} (\nabla \cdot u)\,dx\\
&= -\mathcal{I}_1 + \int_\om (\nabla \cdot u)\Lambda^{d-\alpha} (\nabla \cdot u)\,dx+ \int_\om \lt(\frac{\nabla\rho}{\rho}\cdot u \rt)\Lambda^{d-\alpha} (\nabla \cdot u)\,dx\\
&\le  -\mathcal{I}_1 +\|\Lambda^{\frac{d-\alpha}{2}} (\nabla \cdot u)\|_{L^2}^2+ \frac{\|u\|_{L^\infty}}{1-\|h\|_{L^\infty}} \|\nabla h\|_{L^2} \|\Lambda^{d-\alpha} (\nabla \cdot u)\|_{L^2}\\
&\le  -\mathcal{I}_1 +\|\Lambda^{\frac{d-\alpha}{2}} (\nabla \cdot u)\|_{L^2}^2 + C\|\nabla h\|_{L^2}\|u\|_{H^{m+\frac{d-\alpha}{2}}}^2.
\end{align*}
Here $C=C(m,d,\alpha)$ is a positive constant independent of $T$.

On the other hand, $\mathcal{I}_3$ can be estimated as
\[\begin{aligned}
\mathcal{I}_3 &= -\int_\om \frac1\rho \nabla h \cdot \Lambda^{d-\alpha}\lt( u \cdot \nabla u + u + \Lambda^{\alpha-d} \nabla h\rt)\,dx\\
&\le C\frac{1}{1-\|h\|_{L^\infty}}\|\nabla h\|_{L^2}\|u\|_{H^{m+\frac{d-\alpha}{2}}}^2 + \frac{1}{1-\|h\|_{L^\infty}}\|\nabla h\|_{L^2} \|\Lambda^{d-\alpha} u\|_{L^2} - \frac{1}{1+\|h\|_{L^\infty}}\|\nabla h\|_{L^2}^2\\
&\le C\|\nabla h\|_{L^2}\|u\|_{H^{m+\frac{d-\alpha}{2}}}^2 + 2\|\nabla h\|_{L^2} \|\Lambda^{d-\alpha} u\|_{L^2} - \frac23 \|\nabla h\|_{L^2}^2\\
&\le -\frac12\|\nabla h\|_{L^2}^2 + C\|\nabla h\|_{L^2}\|u\|_{H^{m+\frac{d-\alpha}{2}}}^2 + 6\|\Lambda^{d-\alpha} u\|_{L^2}^2.
\end{aligned}\]

Thus, we combine the estimates for $\mathcal{I}_i$, $i=1,2,3$, to conclude the desired result.
\end{proof}

Before proceeding the higher order estimates, we provide some technical estimates based on the Moser-type inequality below. For the smooth flow of reading, we postpone its proof to Appendix \ref{app_a}.
\begin{lemma}\label{h_decay01}
Let $T>0$, $m > \frac d2+2$, and $(h,u) \in \mc([0,T); X^m)$  be a solution to the system \eqref{np_ER2} satisfying
\[
\sup_{0\le t \le T} \|h(t)\|_{L^\infty} \le \frac12.
\]
 Then for $1 \leq k \leq m-1$ we have
 \begin{itemize}
 \item[(i)] 
 \[
\lt\|\nabla\lt(\frac{\nabla h}{\rho^2} \cdot \pa^k (\rho u) \rt) \rt\|_{L^2} \leq C\lt(1+\|\nabla h\|_{H^{m-1}} \rt)^2 \|\nabla h\|_{H^{m-1}} \|u\|_{H^{m+\frac{d-\alpha}{2}}},
 \]
 \item[(ii)]
 \[
 \lt\|\nabla\lt[ \nabla\cdot\lt( \frac1\rho \lt( \pa^k (\rho u) - (\pa^k \rho) u - \rho (\pa^k u)\rt)\rt)\rt] \rt\|_{L^2} \leq C\lt(1+\|\nabla h\|_{H^{m-1}} \rt)^2 \|\nabla h\|_{H^{m-1}} \|u\|_{H^{m+\frac{d-\alpha}{2}}},
 \]
 \item[(iii)]
 \[
 \lt\| \pa\lt(\frac{1}{\rho^2} (\pa^k h) \nabla h\rt) \rt\|_{L^2} \leq C\lt(1+\|\nabla h\|_{H^{m-1}}\rt)\|\nabla h\|_{H^{m-1}}^2,
 \]
 \item[(iv)]
 \[
 \lt\|\Lambda^{d-\alpha}\pa^{k-1} \pa_t u \rt\|_{L^2} \leq C\lt(\|u\|_{H^{m+\frac{d-\alpha}{2}}}^2 + \|u\|_{H^{m+\frac{d-\alpha}{2}}} + \|\nabla h\|_{H^{m-1}}\rt),
 \]
 and
 \item[(v)]
 \[
 \lt\|\Lambda^{d-\alpha} \pa^{k-1}  (\nabla u : (\nabla  u)^T)\rt\|_{L^2} \leq C\|u\|_{H^{m+\frac{d-\alpha}{2}}}^2.
 \]
 \end{itemize}
 Here $\rho = h + 1$ and $C=C(m,k,d,\alpha)$ is a positive constant independent of $T$.
\end{lemma}

\begin{lemma}\label{h_decay}
Let $T>0$, $m > \frac d2+2$, and $(h,u) \in \mc([0,T); X^m)$  be a solution to the system \eqref{np_ER2} satisfying
\[
\sup_{0\le t \le T} \|h(t)\|_{L^\infty} \le \frac12.
\]
 Then we have
\[\begin{aligned}
\frac{d}{dt}&\int_\om\frac1\rho \pa^k \nabla h \cdot \Lambda^{d-\alpha} \pa^k u\,dx + \frac12 \|\pa^k \nabla h\|_{L^2}^2\\
&\le C(1+\|\nabla h\|_{H^{m-1}})^2 (\|\nabla h\|_{H^{m-1}}^3 + \|u\|_{H^{m+\frac{d-\alpha}{2}}}^3) +\|\Lambda^{\frac{d-\alpha}{2}}\pa^k (\nabla \cdot u)\|_{L^2}^2 + 6\|\Lambda^{d-\alpha}\pa^k u\|_{L^2}^2, 
\end{aligned}\]
for $1\le k\le m-1$, where $C=C(m,k,d,\alpha)$ is a positive constant independent of $T$.
\end{lemma}
\begin{proof} Throughout this proof, $C>0$ denotes the generic constant depending only on $m,k,d$, and $\alpha$, independent of $T$.

Direct computation yields
\[\begin{aligned}
\frac{d}{dt}&\int_\om\frac1\rho \pa^k \nabla h \cdot \Lambda^{d-\alpha} \pa^k u\,dx\\
&=-\int_\om \frac{\pa_t \rho}{\rho^2} \pa^k \nabla h \cdot \Lambda^{d-\alpha} \pa^k u\,dx + \int_\om \frac1\rho \pa^k \nabla \pa_t h\cdot \Lambda^{d-\alpha} \pa^k u\,dx  + \int_\om \frac1\rho \pa^k \nabla h \cdot \Lambda^{d-\alpha} \pa^k \pa_t u\,dx\\
&=: \mathcal{J}_1 + \mathcal{J}_2 + \mathcal{J}_3.
\end{aligned}\]
We estimate $\mathcal{J}_1$, $\mathcal{J}_2$ and $\mathcal{J}_3$ one by one as follows.

\medskip

\noindent $\diamond$ (Estimates for $\mathcal{J}_1$): One obtains
\[\begin{aligned}
\mathcal{J}_1 &= \int_\om \frac{\nabla \cdot (\rho u)}{\rho^2} \pa^k \nabla h \cdot \Lambda^{d-\alpha} \pa^k u\,dx \\
&= \int_\om \frac{\nabla \cdot u}{\rho} \pa^k \nabla h \cdot \Lambda^{d-\alpha} \pa^k u\,dx + \int_\om \frac{\nabla h \cdot u}{\rho^2} \pa^k \nabla h \cdot \Lambda^{d-\alpha} \pa^k u\,dx\\
&\le \frac{\|\nabla u\|_{L^\infty}}{1-\|h\|_{L^\infty}}\|\pa^k \nabla h\|_{L^2}\|\Lambda^{d-\alpha} \pa^k u\|_{L^2} + \frac{\|\nabla h\|_{L^\infty}\|u\|_{L^\infty}}{(1-\|h\|_{L^\infty})^2}\|\pa^k \nabla h\|_{L^2} \|\Lambda^{d-\alpha} \pa^k u\|_{L^2}\\
&\le C(1+\|\nabla h\|_{H^{m-1}})\|\nabla h\|_{H^{m-1}} \|u\|_{H^{m+\frac{d-\alpha}{2}}}^2.
\end{aligned}\]

\noindent $\diamond$ (Estimates for $\mathcal{J}_2$):  We get
\begin{align*}
\mathcal{J}_2&= \int_\om \frac{1}{\rho^2} \pa^k \pa_t h \nabla h \cdot \Lambda^{d-\alpha} \pa^k u\,dx -\int_\om \frac1\rho \pa^k \pa_t h \Lambda^{d-\alpha} \pa^k (\nabla \cdot u)\,dx\\
&=- \int_\om \frac{1}{\rho^2} \pa^k (\nabla \cdot (hu)) \nabla h \cdot \Lambda^{d-\alpha} \pa^k u\,dx -\int_\om \frac{1}{\rho^2} \pa^k (\nabla \cdot u) \nabla h \cdot \Lambda^{d-\alpha} \pa^k u\,dx \\
&\quad + \int_\om \frac1\rho \pa^k (\nabla\cdot(\rho u)) \Lambda^{d-\alpha} \pa^k (\nabla \cdot u)\,dx\\
&=- \int_\om \frac{1}{\rho^2} \pa^k (\nabla \cdot (hu)) \nabla h \cdot \Lambda^{d-\alpha} \pa^k u\,dx -\int_\om \frac{1}{\rho^2} \pa^k (\nabla \cdot u) \nabla h \cdot \Lambda^{d-\alpha} \pa^k u\,dx \\
&\quad + \int_\om \frac{\nabla h}{\rho^2} \cdot \pa^k (\rho u) \Lambda^{d-\alpha}\pa^k(\nabla \cdot u)\,dx - \int_\om \frac1\rho \pa^k (\rho u)\cdot \nabla \Lambda^{d-\alpha} \pa^k (\nabla \cdot u)\,dx\\
&=- \int_\om \frac{1}{\rho^2} \pa^k (\nabla \cdot (hu)) \nabla h \cdot \Lambda^{d-\alpha} \pa^k u\,dx -\int_\om \frac{1}{\rho^2} \pa^k (\nabla \cdot u) \nabla h \cdot \Lambda^{d-\alpha} \pa^k u\,dx \\
&\quad - \int_\om \nabla\lt(\frac{\nabla h}{\rho^2} \cdot \pa^k (\rho u) \rt)\cdot\Lambda^{d-\alpha}\pa^k u\,dx - \int_\om \frac1\rho (\pa^k h) u\cdot \nabla \Lambda^{d-\alpha} \pa^k (\nabla \cdot u)\,dx\\
&\quad -\int_\om \pa^k u \cdot \nabla \Lambda^{d-\alpha} \pa^k (\nabla \cdot u)\,dx -\int_\om \frac1\rho \lt( \pa^k(\rho u) -(\pa^k \rho)u - \rho (\pa^k u)\rt)\cdot \nabla \Lambda^{d-\alpha} \pa^k (\nabla \cdot u)\,dx\\
&=: \sum_{i=1}^6\mathcal{J}_{2i}.
\end{align*}
For $\mathcal{J}_{21}$, we use Moser-type inequality in Proposition \ref{ineqs} to get
\begin{align*}
\mathcal{J}_{21} &= -\int_\om \frac{1}{\rho^2} \pa^k (\nabla h \cdot u) \nabla h \cdot \Lambda^{d-\alpha} \pa^k u\,dx -\int_\om \frac{1}{\rho^2} \lt(\pa^k (h \nabla \cdot u) - h\pa^k (\nabla \cdot u)\rt) \nabla h \cdot \Lambda^{d-\alpha} \pa^k u\,dx \\
&\quad -  \int_\om \frac{1}{\rho^2} h\pa^k (\nabla \cdot u) \nabla h \cdot \Lambda^{d-\alpha} \pa^k u\,dx \\
&\le  \frac{\|\nabla h\|_{L^\infty}}{(1-\|h\|_{L^\infty})^2}\lt(\|\pa^k (\nabla h \cdot u))\|_{L^2} + \|\pa^k ( h \nabla \cdot u)) - h \pa^k (\nabla \cdot u)\|_{L^2}\rt)\|\Lambda^{d-\alpha} \pa^k u\|_{L^2}\\
&\quad -  \int_\om \frac{1}{\rho^2} h\pa^k (\nabla \cdot u) \nabla h \cdot \Lambda^{d-\alpha} \pa^k u\,dx \\
&\le C\|\nabla h\|_{L^\infty} \lt( \|\nabla h\|_{L^\infty} \|\pa^{k} u\|_{H^1} + \|u\|_{W^{1,\infty}}\|\pa^{k} h\|_{H^1}\rt)\|\Lambda^{d-\alpha} \pa^k u\|_{L^2}\\
&\quad -  \int_\om \frac{1}{\rho^2} h\pa^k (\nabla \cdot u) \nabla h \cdot \Lambda^{d-\alpha} \pa^k u\,dx \\
&\le C\|\nabla h\|_{H^{m-1}}^2 \|u\|_{H^{m+\frac{d-\alpha}{2}}}^2-  \int_\om \frac{1}{\rho^2} h\pa^k (\nabla \cdot u) \nabla h \cdot \Lambda^{d-\alpha} \pa^k u\,dx.
\end{align*}
This implies
\[\begin{aligned}
\mathcal{J}_{21} +\mathcal{J}_{22} &\le  C\|\nabla h\|_{H^{m-1}}^2 \|u\|_{H^{m+\frac{d-\alpha}{2}}}^2-  \int_\om \frac{1}{\rho} \pa^k (\nabla \cdot u) \nabla h \cdot \Lambda^{d-\alpha} \pa^k u\,dx \le \|\nabla h\|_{H^{m-1}}^2 \|u\|_{H^{m+\frac{d-\alpha}{2}}}^2.
\end{aligned}\]
For $\mathcal{J}_{23}$, we use H\"older's inequality and Lemma \ref{h_decay01} (i) to obtain
\[
\mathcal{J}_{23} \leq \lt\|\nabla\lt(\frac{\nabla h}{\rho^2} \cdot \pa^k (\rho u) \rt) \rt\|_{L^2}\|\Lambda^{d-\alpha} \pa^k u\|_{L^2} \leq C\lt(1+\|\nabla h\|_{H^{m-1}} \rt)^2 \|\nabla h\|_{H^{m-1}} \|u\|_{H^{m+\frac{d-\alpha}{2}}}^2.
\]
We next use the integration by parts to deduce
\[
\mathcal{J}_{25} = \|\Lambda^{\frac{d-\alpha}{2}} \pa^k (\nabla \cdot u)\|_{L^2}^2.
\]
For the estimate of $\mathcal{J}_{26}$, we use Lemma \ref{h_decay01} (ii) to deduce
\begin{align*}
\mathcal{J}_{26} &= -\int_\om \nabla\lt[ \nabla\cdot\lt( \frac1\rho \lt( \pa^k (\rho u) - (\pa^k \rho) u - \rho (\pa^k u)\rt)\rt)\rt] \cdot \Lambda^{d-\alpha}\pa^k u\,dx\\
&\le \lt\|\nabla\lt[ \nabla\cdot\lt( \frac1\rho \lt( \pa^k (\rho u) - (\pa^k \rho) u - \rho (\pa^k u)\rt)\rt)\rt] \rt\|_{L^2}  \|\Lambda^{d-\alpha} \pa^k u\|_{L^2}\\
&\le C\lt( 1+ \|\nabla h\|_{H^{m-1}}\rt)^2\|\nabla h\|_{H^{m-1}}\|u\|_{H^{m+\frac{d-\alpha}{2}}}^2.
\end{align*}

 Hence, we gather the estimates for $\mathcal{J}_{2i}$'s to yield
\[
\begin{aligned}
\mathcal{J}_2 &\le - \int_\om \frac1\rho (\pa^k h) u\cdot \nabla \Lambda^{d-\alpha} \pa^k (\nabla \cdot u)\,dx + \|\Lambda^{\frac{d-\alpha}{2}} \pa^k (\nabla \cdot u)\|_{L^2}^2 \cr
&\quad + C\lt( 1+ \|\nabla h\|_{H^{m-1}}\rt)^2\|\nabla h\|_{H^{m-1}}\|u\|_{H^{m+\frac{d-\alpha}{2}}}^2.
\end{aligned}
\]

\medskip

\noindent $\diamond$ (Estimates for $\mathcal{J}_3$): In this case,
\begin{align*}
\mathcal{J}_3 &= \int_\om \frac{1}{\rho^2} (\pa^k h) \nabla h \cdot \Lambda^{d-\alpha} \pa^k \pa_t u\,dx - \int_\om \frac1\rho \pa^k h \Lambda^{d-\alpha} \pa^k \pa_t (\nabla \cdot u)\,dx\\
&= \int_\om \frac{1}{\rho^2} (\pa^k h) \nabla h \cdot \Lambda^{d-\alpha} \pa^k \pa_t u\,dx  +\int_\om \frac1\rho (\pa^k h) \Lambda^{d-\alpha}\pa^k (\nabla \cdot (u\cdot \nabla u) + \nabla \cdot u + \Lambda^{\alpha-d} \Delta h)\,dx\\
&= \int_\om \frac{1}{\rho^2} (\pa^k h) \nabla h \cdot \Lambda^{d-\alpha} \pa^k \pa_t u\,dx + \int_\om \frac1\rho (\pa^k h)  \Lambda^{d-\alpha}\pa^k (\nabla \cdot (u \cdot \nabla u))\,dx\\
&\quad + \int_\om \frac1\rho (\pa^k h)\Lambda^{d-\alpha}\pa^k (\nabla \cdot u)\,dx + \int_\om \frac1\rho (\pa^k h)\Delta (\pa^k h)\,dx\\
&=: \sum_{i=1}^4 \mathcal{J}_{3i}.
\end{align*}
For the term $\mathcal{J}_{31}$, we use Lemma \ref{h_decay01} (iii) \& (iv) to have
\begin{align*}
\mathcal{J}_{31} &= \int_\om \pa\lt(\frac{1}{\rho^2} (\pa^k h) \nabla h\rt) \cdot \Lambda^{d-\alpha}\pa^{k-1} \pa_t u\,dx\\
&\leq \lt\| \pa\lt(\frac{1}{\rho^2} (\pa^k h) \nabla h\rt) \rt\|_{L^2}  \lt\|\Lambda^{d-\alpha}\pa^{k-1} \pa_t u\rt\|_{L^2} \cr
&\le C\lt(1+\|\nabla h\|_{H^{m-1}}\rt)\|\nabla h\|_{H^{m-1}}^2\lt(\|u\|_{H^{m+\frac{d-\alpha}{2}}}^2 + \|u\|_{H^{m+\frac{d-\alpha}{2}}} + \|\nabla h\|_{H^{m-1}}\rt)\\
&\le C \lt(1+\|\nabla h\|_{H^{m-1}} + \|u\|_{H^{m+\frac{d-\alpha}{2}}}\rt)\|\nabla h\|_{H^{m-1}}^2\lt(\|u\|_{H^{m+\frac{d-\alpha}{2}}}  + \|\nabla h\|_{H^{m-1}}\rt).
\end{align*}

We then estimate $\mathcal{J}_{32}$ as 
\begin{align*}
\mathcal{J}_{32} &= -\int_\om \pa\lt( \frac1\rho (\pa^k h)\rt) \Lambda^{d-\alpha} \pa^{k-1}( \nabla \cdot (u \cdot \nabla u))\,dx\\
&= -\int_\om \pa\lt( \frac1\rho (\pa^k h)\rt) \Lambda^{d-\alpha} \pa^{k-1}( (u \cdot \nabla) (\nabla \cdot u))\,dx\\
&\quad -\int_\om \pa\lt( \frac1\rho (\pa^k h)\rt) \Lambda^{d-\alpha} \pa^{k-1} (\nabla \cdot (u \cdot \nabla u)-u \cdot \nabla (\nabla \cdot u))\,dx\\
&= -\int_\om \pa\lt( \frac1\rho (\pa^k h)\rt) u\cdot \nabla \Lambda^{d-\alpha} \pa^{k-1} (\nabla \cdot u)\,dx - \int_\om \pa\lt( \frac1\rho (\pa^k h)\rt) [\Lambda^{d-\alpha}, u \cdot \nabla]\pa^{k-1}(\nabla \cdot u)\,dx\\
&\quad - \int_\om \pa\lt( \frac1\rho (\pa^k h)\rt) \Lambda^{d-\alpha}\lt[\pa^{k-1} (u \cdot \nabla(\nabla \cdot u)) - u \cdot \nabla \pa^{k-1}(\nabla \cdot u) \rt]\,dx\\
&\quad  -\int_\om \pa\lt( \frac1\rho (\pa^k h)\rt) \Lambda^{d-\alpha} \pa^{k-1} (\nabla \cdot (u \cdot \nabla u)-u \cdot \nabla (\nabla \cdot u))\,dx\\
&=:\sum_{i=1}^4 \mathcal{J}_{32i}.
\end{align*}
First, we estimate $\mathcal{J}_{321}$ as
\begin{align*}
\mathcal{J}_{321} &=  \int_\om  \frac1\rho (\pa^k h)\pa u\cdot \nabla \Lambda^{d-\alpha} \pa^{k-1} (\nabla \cdot u)\,dx + \int_\om \frac1\rho (\pa^k h) u\cdot \nabla \Lambda^{d-\alpha} \pa^{k} (\nabla \cdot u)\,dx\\
&=  -\int_\om  \nabla \cdot \lt(\frac1\rho (\pa^k h)\pa u\rt) \Lambda^{d-\alpha} \pa^{k-1} (\nabla \cdot u)\,dx + \int_\om \frac1\rho (\pa^k h) u\cdot \nabla \Lambda^{d-\alpha} \pa^{k} (\nabla \cdot u)\,dx\\
&\le C\lt( \frac{\|\nabla h\|_{L^\infty} \|\pa^k h\|_{L^2} \|\pa u\|_{L^\infty} }{(1-\|h\|_{L^\infty})^2} + \frac{\|\pa^k\nabla h\|_{L^2}\|\pa u\|_{L^\infty} + \|\pa^k h\|_{L^2}\|\nabla^2 u\|_{L^\infty}}{1-\|h\|_{L^\infty}}\rt) \|\Lambda^{d-\alpha} \pa^{k-1} (\nabla \cdot u)\|_{L^2}\\
&\quad + \int_\om \frac1\rho (\pa^k h) u\cdot \nabla \Lambda^{d-\alpha} \pa^{k} (\nabla \cdot u)\,dx\\
&\le C(1+\|\nabla h\|_{H^{m-1}}) \|\nabla h\|_{H^{m-1}}\|u\|_{H^{m+\frac{d-\alpha}{2}}}^2+ \int_\om \frac1\rho (\pa^k h) u\cdot \nabla \Lambda^{d-\alpha} \pa^{k} (\nabla \cdot u)\,dx.
\end{align*}
Next, Lemma \ref{comm_est2} yields
\[\begin{aligned}
\mathcal{J}_{322} &\le C\lt(\frac{\|\nabla h\|_{L^\infty} \|\pa^k h\|_{L^2}}{(1-\|h\|_{L^\infty})^2} + \frac{\|\pa^{k+1}h\|_{L^2}}{1-\|h\|_{L^\infty}} \rt)\|u\|_{H^{\frac{d}{2}+1+(d-\alpha)+\e }} \|\pa^{k-1} (\nabla \cdot u)\|_{H^{d-\alpha}}\\
&\le C(1+\|\nabla h\|_{H^{m-1}}) \|\nabla h\|_{H^{m-1}}\|u\|_{H^{m+\frac{d-\alpha}{2}}}^2,
\end{aligned}
\]
where $\e$ satisfies $(d-\alpha)/2 + \e \le 1$ so that $\frac{d}{2}+1+(d-\alpha)+\e \le m+\frac{d-\alpha}{2}$. \\

\noindent For $\mathcal{J}_{323}$, we need to estimate
\[
\|\Lambda^{d-\alpha} (\pa^\ell u \cdot \nabla \pa^{k-1-\ell} (\nabla \cdot u))\|_{L^2},
\]
for $1 \le \ell\le k-1$. For $\ell=1$, 

\[
\begin{aligned}
\|\Lambda^{d-\alpha} \pa u \cdot \nabla \pa^{k-2} (\nabla \cdot u)\|_{L^2} &\le \|\pa u \cdot \nabla \pa^{k-2} (\nabla \cdot u)\|_{H^{d-\alpha}}\\
&\le C\|\widehat{\nabla u}\|_{L_\xi^1}\|u\|_{H^{k+(d-\alpha)}} \le C\|u\|_{H^{m+\frac{d-\alpha}{2}}}^2.
\end{aligned}
\]
For $2 \le \ell \le k-1$, we use Proposition \ref{ineqs} to get
\[\begin{aligned}
\|\Lambda^{d-\alpha} \pa^\ell u \cdot \nabla \pa^{k-1-\ell} (\nabla \cdot u)\|_{L^2} &\le \|\pa^\ell u \cdot \nabla \pa^{k-1-\ell} (\nabla \cdot u)\|_{H^2}\\
&\le C \|\pa^\ell u\|_{H^{m-\ell}} \|\pa^{k-\ell+1}u\|_{H^{m-(k-\ell+1)}}\\
&\le C\|u\|_{H^{m+\frac{d-\alpha}{2}}}^2,
\end{aligned}\]
and this implies
\[\begin{aligned}
\mathcal{J}_{323}&\le C\lt(\frac{\|\nabla h\|_{L^\infty} \|\pa^k h\|_{L^2}}{(1-\|h\|_{L^\infty})^2} + \frac{\|\pa^{k+1}h\|_{L^2}}{1-\|h\|_{L^\infty}} \rt)\|u\|_{H^{m+\frac{d-\alpha}{2}}}^2\\
& \le C(1+\|\nabla h\|_{H^{m-1}}) \|\nabla h\|_{H^{m-1}}\|u\|_{H^{m+\frac{d-\alpha}{2}}}^2.
\end{aligned}\]
Now, for $\mathcal{J}_{324}$, we use the following estimate
\[
\lt\| \pa\lt( \frac1\rho (\pa^k h)\rt) \rt\|_{L^2} \leq C\lt(\frac{\|\nabla h\|_{L^\infty} \|\pa^k h\|_{L^2}}{(1-\|h\|_{L^\infty})^2} + \frac{\|\pa^{k+1}h\|_{L^2}}{1-\|h\|_{L^\infty}} \rt) \leq C(1+\|\nabla h\|_{H^{m-1}}) \|\nabla h\|_{H^{m-1}}
\]
together with Lemma \ref{h_decay01} (v) to have
\begin{align*}
\mathcal{J}_{324} &= -\sum_{i,j=1}^d \int_\om  \pa\lt( \frac1\rho (\pa^k h)\rt) \Lambda^{d-\alpha} \pa^{k-1}  (\pa_{x_i} u_j \pa_{x_j} u_i)\,dx\\
&= -   \int_\om  \pa\lt( \frac1\rho (\pa^k h)\rt) \Lambda^{d-\alpha} \pa^{k-1}  (\nabla u : (\nabla  u)^T)\,dx\\
&\leq \lt\| \pa\lt( \frac1\rho (\pa^k h)\rt) \rt\|_{L^2} \|\Lambda^{d-\alpha} \pa^{k-1}  (\nabla u : (\nabla  u)^T)\|_{L^2}\cr
&\le C(1+\|\nabla h\|_{H^{m-1}}) \|\nabla h\|_{H^{m-1}}\|u\|_{H^{m+\frac{d-\alpha}{2}}}^2.
\end{align*} 

Then we collect the estimates for $\mathcal{J}_{32i}$'s to yield
\[
\mathcal{J}_{32} \le C(1+\|\nabla h\|_{H^{m-1}}) \|\nabla h\|_{H^{m-1}}\|u\|_{H^{m+\frac{d-\alpha}{2}}}^2 + \int_\om \frac1\rho (\pa^k h) u\cdot \nabla \Lambda^{d-\alpha} \pa^{k} (\nabla \cdot u)\,dx.
\]
For $\mathcal{J}_{33}$ and $\mathcal{J}_{34}$,
\[\begin{aligned}
\mathcal{J}_{33} &= \int_\om \frac{1}{\rho^2}(\pa^k h) \nabla h \cdot \Lambda^{d-\alpha}\pa^k u\,dx - \int_\om \frac1\rho \pa^k \nabla h \cdot \Lambda^{d-\alpha} \pa^k u\,dx\\
&\le \frac{\|\nabla h\|_{L^\infty}}{(1-\|h\|_{L^\infty})^2}\|\pa^k h\|_{L^2}\|\Lambda^{d-\alpha} \pa^k u\|_{L^2} + \frac{1}{1-\|h\|_{L^\infty}} \|\pa^k \nabla h\|_{L^2}\|\Lambda^{d-\alpha} \pa^k u\|_{L^2}\\
&\le C\|\nabla h\|_{H^{m-1}}^2 \|u\|_{H^{m+\frac{d-\alpha}{2}}} + \frac16 \|\pa^k \nabla h\|_{L^2}^2 + 6\|\Lambda^{d-\alpha} \pa^k u\|_{L^2}^2,\\
\mathcal{J}_{34} &= \int_\om \frac{1}{\rho^2}(\pa^k h) \nabla h \cdot \pa^k \nabla h\,dx - \int_\om \frac1\rho |\pa^k \nabla h|^2\,dx\\
&\le \frac{\|\nabla h\|_{L^\infty}}{(1-\|h\|_{L^\infty})^2}\|\pa^k h\|_{L^2}\|\pa^k \nabla h\|_{L^2} - \frac{1}{1+\|h\|_{L^\infty}}\|\pa^k \nabla h\|_{L^2}^2\\
&\le C\|\nabla h\|_{H^{m-1}}^3 -\frac{2}{3}\|\pa^k \nabla h\|_{L^2}^2.
\end{aligned}
\]
Thus, we combine the estimates for $\mathcal{J}_{3i}$'s to obtain
\[\begin{aligned}
\mathcal{J}_3 &\le C(1+\|\nabla h\|_{H^{m-1}})(\|\nabla h\|_{H^{m-1}}^3 + \|u\|_{H^{m+\frac{d-\alpha}{2}}}^3) -\frac12\|\pa^k \nabla h\|_{L^2}^2 + 6\|\Lambda^{d-\alpha} \pa^k u\|_{L^2}^2\\
&\quad + \int_\om \frac1\rho (\pa^k h) u\cdot \nabla \Lambda^{d-\alpha} \pa^{k} (\nabla \cdot u)\,dx.
\end{aligned}\]
Therefore, we gather all the results for $\mathcal{J}_i$'s to have
\[\begin{aligned}
\frac{d}{dt}&\int_\om\frac1\rho \pa^k \nabla h \cdot \Lambda^{d-\alpha} \pa^k u\,dx + \frac12 \|\pa^k \nabla h\|_{L^2}^2\\
&\le C(1+\|\nabla h\|_{H^{m-1}})^2 (\|\nabla h\|_{H^{m-1}}^3 + \|u\|_{H^{m+\frac{d-\alpha}{2}}}^3) +\|\Lambda^{\frac{d-\alpha}{2}}\pa^k (\nabla \cdot u)\|_{L^2}^2 + 6\|\Lambda^{d-\alpha}\pa^k u\|_{L^2}^2.
\end{aligned}\]
\end{proof}

Based on the results up to now, we provide the uniform-in-time bound estimate of solutions below.
\begin{proposition}\label{P2.1}
Let $T>0$, $m > \frac d2+2$, and $(h,u) \in \mc([0,T); X^m)$  be a solution to the system \eqref{np_ER2}. Suppose that $X(T;m)\le \e_0^2 \ll 1$ so that
\[
\sup_{0\le t \le T} \|h(t)\|_{L^\infty} \le \frac12.
\]
 Then there exists a positive constant $C^*$ independent of $T$ such that
\[
X(T;m) \le C^*X_0(m).
\]
 \end{proposition}
  \begin{proof}  
Applying Lemma \ref{h_decay} and \eqref{equiv_u} implies that we can find $C_1>0$ independent of $T$ such that
\begin{align}\label{est_mix}
\begin{aligned}
\frac{d}{dt}&\lt(\sum_{0 \le k \le m-1} \int_\om \frac1\rho \pa^k \nabla h \cdot \Lambda^{d-\alpha} \pa^k u\,dx \rt) + \frac12\|\nabla h\|_{H^{m-1}}^2\\
&\le C(1+\|\nabla h\|_{H^{m-1}})^2 (\|\nabla h\|_{H^{m-1}}^3 + \|u\|_{H^{m+\frac{d-\alpha}{2}}}^3) \cr
&\quad +2\sum_{0\le k \le m-1}\|(\nabla \cdot U_k)\|_{L^2}^2 + 6\sum_{0\le k \le m-1}\|\Lambda^{d-\alpha}\pa^k u\|_{L^2}^2\\
&\le C(1+\|\nabla h\|_{H^{m-1}})^2 (\|\nabla h\|_{H^{m-1}}^3 + \|u\|_{H^{m+\frac{d-\alpha}{2}}}^3) + C_1 \sum_{0\le k \le m}\|U_k\|_{L^2}^2,
\end{aligned}
\end{align}
where we used
\[
\|\Lambda^{d-\alpha}\pa^k u\|_{L^2} = \|\Lambda^{\frac{d-\alpha}{2}}U_k\|_{L^2} \leq \|U_k\|_{H^1}.
\]
On the other hand, it follows from Lemmas \ref{est_u} and \ref{est_h} that 
\begin{align}\label{est_mix2}
\begin{aligned}
&\frac{d}{dt}\lt[\sum_{1 \le k \le m} \lt(\|U_k\|_{L^2}^2 +\lt\|\frac{1}{\sqrt{\rho}} \pa^k h\rt\|_{L^2}^2\rt)\rt] + 2\sum_{1 \le k \le m} \|U_k\|_{L^2}^2 \cr
&\quad \leq C(1+\|\nabla h\|_{H^{m-1}})^{2(m-1)} (\|\nabla h\|_{H^{m-1}}^3 + \|u\|_{H^{m+\frac{d-\alpha}{2}}}^3).
\end{aligned}
\end{align}
Here we used
\[
\lt( 1+\|\nabla \log \rho\|_{L^\infty}\rt)^{2(k-1)}\sum_{0<l\le k} \lt\| \frac{1}{\sqrt{\rho}}R_l\rt\|_{L^2}^2 \leq C\lt( 1+\|\nabla h\|_{H^{m-1}}\rt)^{2(m-1)} \|\nabla h\|_{H^{k-1}}^2,
\]
for $k=1,\dots,m-1$, where $C > 0$ is independent of $t$.\\

\noindent Now, we use Lemma \ref{est_u} for $k=0$ and Lemma \ref{est_h_0} to get
$$\begin{aligned}
\frac{d}{dt}\lt(\|h\|_{L^2}^2 + \|U_0\|_{L^2}^2 \rt) + 2\|U_0\|_{L^2}^2 &\leq C\|u\|_{H^{m+\frac{d-\alpha}{2}}}^3 + 2\|u\|_{L^2}\|\nabla h\|_{L^2}\|h\|_{L^\infty},
\end{aligned}$$
and combine this with Proposition \ref{zeroth} to obtain
$$\begin{aligned}
&\frac{d}{dt} \lt( \int_\om \rho |u|^2\,dx +  \|h\|_{\dot{H}^{-\frac{d-\alpha}{2}}}^2 + \|h\|_{L^2}^2 + \|U_0\|_{L^2}^2 \rt) + \|u\|_{L^2}^2  + 2\|U_0\|_{L^2}^2 \cr
&\quad \leq C\|u\|_{H^{m+\frac{d-\alpha}{2}}}^3 + 2\|u\|_{L^2}\|\nabla h\|_{L^2}\|h\|_{L^\infty}.
\end{aligned}$$
We next choose a positive constant $\eta_1 \ll 1$ satisfying $C_1\eta_1 <1$ and combine \eqref{est_mix} and \eqref{est_mix2} to find
\begin{align}\label{est_he}
\begin{aligned}
\frac{d}{dt}&\Bigg[\int_\om \rho |u|^2\,dx + \|h\|_{\dot{H}^{-\frac{d-\alpha}{2}}}^2 + \|h\|_{L^2}^2 + \|U_0\|_{L^2}^2 +\sum_{1 \le k \le m} \lt(\|U_k\|_{L^2}^2 +\lt\|\frac{1}{\sqrt{\rho}} \pa^k h\rt\|_{L^2}^2\rt) \\
&\qquad+\eta_1 \sum_{0 \le k \le m-1} \int_\om \frac1\rho \pa^k \nabla h \cdot \Lambda^{d-\alpha} \pa^k u\,dx  \Bigg]+  \frac{\eta_1}2 \|\nabla h\|_{H^{m-1}}^2 + \lt(\|u\|_{L^2}^2 + \sum_{1 \le k \le m} \|U_k\|_{L^2}^2\rt)\\\
&\le C(1+\|\nabla h\|_{H^{m-1}})^{2} \lt(\|\nabla h\|_{H^{m-1}}^3 + \|u\|_{H^{m+\frac{d-\alpha}{2}}}^3\rt) + C\|h\|_{H^m}\lt(\|u\|_{L^2}^2 + \|\nabla h\|_{L^2}^2\rt).
\end{aligned}
\end{align}
Since we have the following equivalence relations
\bq\label{equi1}
\int_\om \rho |u|^2\,dx \approx \|u\|_{L^2}^2, \quad \|\nabla h\|_{H^{m-1}} \approx \sum_{1 \le k \le m}\lt\|\frac{1}{\sqrt{\rho}} \pa^k h\rt\|_{L^2}, \quad \|u\|_{H^{m+\frac{d-\alpha}{2}}}\approx \|u\|_{L^2} + \sum_{0 \le k \le m} \|U_k\|_{L^2}, 
\eq
and
\begin{align}\label{equi2}
\begin{aligned}
&\sum_{0 \le k \le m} \|U_k\|_{L^2}^2 + \sum_{1 \le k \le m}\lt\|\frac{1}{\sqrt{\rho}} \pa^k h\rt\|_{L^2}^2 +\eta_1 \sum_{0 \le k \le m-1} \int_\om \frac1\rho \pa^k \nabla h \cdot \Lambda^{d-\alpha} \pa^k u\,dx \\
&\hspace{2cm}\approx \sum_{0 \le k \le m} \|U_k\|_{L^2}^2 + \sum_{1 \le k \le m}\lt\|\frac{1}{\sqrt{\rho}} \pa^k h\rt\|_{L^2}^2,
\end{aligned}
\end{align}
for $\eta_1 > 0$ sufficiently small, we can find positive constants $\eta_2$ and $C_2$ independent of $T$ such that
\begin{align}\label{est_energy}
\begin{aligned}
\frac{d}{dt}\lt( \mathcal{Y}^m +\|h\|_{L^2}^2+\|h\|_{\dot{H}^{-\frac{d-\alpha}{2}}}^2\rt) + \eta_2 \mathcal{Y}^m &\le C_2\e_0 \mathcal{Y}^m,
\end{aligned}
\end{align}
where $\mathcal{Y}^m= \mathcal{Y}^m(t)$ is given by
\[
\mathcal{Y}^m:=\int_\om \rho |u|^2\,dx + \sum_{0 \le k \le m} \|U_k\|_{L^2}^2 +\sum_{1 \le k \le m}\lt\|\frac{1}{\sqrt{\rho}} \pa^k h\rt\|_{L^2}^2 +\eta_1 \sum_{0 \le k \le m-1} \int_\om \frac1\rho \pa^k \nabla h \cdot \Lambda^{d-\alpha} \pa^k u\,dx.
\]
Thus, once $\e_0$ is chosen sufficiently small so that $\eta_2 - C_2\e_0>0$, then we set $\lambda:= \eta_2 - C_2\e_0$ and use Gr\"onwall's lemma to get
\[
\mathcal{Y}^m(t) +\|h(\cdot,t)\|_{L^2}^2+\|h(\cdot,t)\|_{\dot{H}^{-\frac{d-\alpha}{2}}}^2 + \lambda\int_0^t \mathcal{Y}^m(\tau)\,d\tau \le  \mathcal{Y}^m(0)  +\|h_0\|_{L^2}^2+\|h_0\|_{\dot{H}^{-\frac{d-\alpha}{2}}}^2.
\]
Since 
\[
\mathcal{Y}^m(t) +\|h(\cdot,t)\|_{L^2}^2+\|h(\cdot,t)\|_{\dot{H}^{-\frac{d-\alpha}{2}}}^2  \approx \|(h,u)(\cdot,t)\|_{X^m}^2
\]
for $\eta_1 > 0$ small enough, we conclude the desired result.
\end{proof}

\subsection{Proof of Theorem \ref{main_thm} } We are now ready to provide the details of proof of Theorem \ref{main_thm}. 

 First, choose $\e_0$ as required in Proposition \ref{P2.1}.  Then we set $\e_1$ as
\[
\e_1^2 := \frac{\e_0^2}{2(1+C^*)}.
\]
By the local existence theory, Theorem \ref{thm_local}, we can find $T_0>0$ such that if the initial data $(h_0, u_0)$ satisfies $X_0(m)< \e_1^2$, a solution $(h,u)$ to \eqref{np_ER2} exists in $\mc([0,T_0); X^m)$. Assume for a contradiction that
\[
T^* := \sup\{ T>0 \ | \ X(T;m) \le \e_0^2\} <\infty.
\]
Then by definition,

\[
\e_0^2 = X(T^*;m) \le C^*X_0(m) < C^*\e_1^2 = \frac{C^*}{2(1+C^*)}\e_0^2 < \e_0^2,
\]
and this contradicts the assumption. Thus, the solution exists in $\mc(\R_+; X^m)$.

%
%
%
%
%
\section{Large-time behavior of solutions}\label{sec_lt}

%
%
%
%
%
\subsection{Whole space case: algebraic decay rate of convergence}

To get the large-time behavior estimates for the whole space case, we investigate negative Sobolev norms. First, we present an auxiliary lemma below.

\begin{lemma}\label{neg_sobo}
\begin{enumerate}
\item[(i)]
Let $-d<s_1 < s<s_2<d$ and $f \in (\dot{H}^{s_1} \cap \dot{H}^{s_2})(\R^d)$. Then we have
\[
\|f\|_{\dot{H}^s} \le \|f\|_{\dot{H}^{s_1}}^{\frac{s_2-s}{s_2-s_1}} \|f\|_{\dot{H}^{s_2}}^{\frac{s-s_1}{s_2-s_1}}.
\]
\item[(ii)]
If $s\in (0, d)$, $1<p<q<\infty$ and $1/q + s/d= 1/p$, then we have
\[
\|\Lambda^{-s}f\|_{L^q} \lesssim_p \|f\|_{L^p}. 
\]
\end{enumerate}
\end{lemma}
\begin{proof}
For (i), since we have
\[
s = \frac{s_2 -s}{s_2-s_1}s_1 + \frac{s-s_1}{s_2-s_1} s_2 ,
\]
we use H\"older's inequality to obtain
\[\begin{aligned}
\int_{\R^d} |\xi|^{2s} |\hat{f}(\xi)|^2\,d\xi &= \int_{\R^d} |\xi|^{2s_1 \lt( \frac{s_2 -s}{s_2-s_1}\rt)}|\xi|^{2s_2 \lt( \frac{s-s_1}{s_2-s_1} \rt)} |\hat{f}(\xi)|^{2\lt( \frac{s_2 -s}{s_2-s_1}\rt)} |\hat{f}(\xi)|^{2\lt( \frac{s-s_1}{s_2-s_1} \rt)}\,d\xi\\
&\le \lt(\int_{\R^d} |\xi|^{2s_1} |\hat{f}(\xi)|^2\,d\xi \rt)^{ \frac{s_2 -s}{s_2-s_1}} \lt(\int_{\R^d} |\xi|^{2s_2} |\hat{f}(\xi)|^2\,d\xi \rt)^{ \frac{s-s_1}{s_2-s_1}},
\end{aligned}
\]
and this implies the desired result. 

The inequality in (ii) is the well-known Hardy--Littlewood--Sobolev inequality, and for the proof, we refer to \cite[p.119, Theorem 1]{St70}.
\end{proof}

\begin{lemma}\label{neg_est}
Let $T>0$, $m > \frac d2+2$, and $ 0< s\le \frac{\alpha}{2}$. Let $(h,u) \in \mc([0,T); X^m)$  be a solution to the system \eqref{np_ER2} satisfying
\[
\sup_{0\le t \le T} \|h(t)\|_{L^\infty} \le \frac12.
\]
 Then we have
\[\begin{aligned}
\frac{d}{dt}&\lt(\|\Lambda^{-s} u\|_{L^2}^2 + \|\Lambda^{-s-\frac{d-\alpha}{2}} h\|_{L^2}^2\rt) + 2\|\Lambda^{-s}u\|_{L^2}^2 \\
&\le C\|u\|_{H^{m+\frac{d-\alpha}{2}}}^2\|\Lambda^{-s}u\|_{L^2} + C\|h\|_{L^2}\|u\|_{H^{m+\frac{d-\alpha}{2}}}\|\Lambda^{-s-\frac{d-\alpha}{2}} \nabla h\|_{L^2},
\end{aligned}\]
where $C=C(s,\alpha,d,m)$ is a positive constant independent of $T$.
\end{lemma}
\begin{proof}
Direct estimate gives
\[\begin{aligned}
\frac12&\frac{d}{dt}\lt(\|\Lambda^{-s} u\|_{L^2}^2 + \|\Lambda^{-s-\frac{d-\alpha}{2}} h\|_{L^2}^2\rt) + \|\Lambda^{-s}u\|_{L^2}^2\\
&= -\int_{\R^d} \Lambda^{-s}(u \cdot \nabla u) \cdot \Lambda^{-s}u\,dx - \int_{\R^d} \Lambda^{-s-\frac{d-\alpha}{2}}(\nabla \cdot (h  u))  \Lambda^{-s-\frac{d-\alpha}{2}} h\,dx\\
&=  - \int_{\R^d} \Lambda^{-s}(u \cdot \nabla u) \cdot \Lambda^{-s}u\,dx + \int_{\R^d} \Lambda^{-s-\frac{d-\alpha}{2}} (h  u)\cdot  \Lambda^{-s-\frac{d-\alpha}{2}} \nabla h\,dx \\
&=: \mathcal{I}_1 + \mathcal{I}_2. 
\end{aligned}\]
For $\mathcal{I}_1$, one gets
\[\begin{aligned}
\mathcal{I}_1 &\le \|\Lambda^{-s} (u \cdot \nabla u)\|_{L^2}\|\Lambda^{-s}u\|_{L^2}\\
&\le \|u \cdot \nabla u\|_{L^{\frac{1}{\frac{1}{2} + \frac{s}{d}}}}\|\Lambda^{-s}u\|_{L^2}\\
&\le \|\nabla u\|_{L^2} \|u\|_{L^{\frac{d}{s}}}\|\Lambda^{-s}u\|_{L^2}\\
&\le C\|\nabla u\|_{L^2} \|\nabla^{\lt[\frac{d}{2} \rt] + 1-s} u\|_{L^2}^{\theta} \| u\|_{L^2}^{1-\theta}\|\Lambda^{-s}u\|_{L^2}\\
&\le C\|u\|_{H^{m+\frac{d-\alpha}{2}}}^2 \|\Lambda^{-s}u\|_{L^2},
\end{aligned}\]
where we used Gagliardo--Nirenberg interpolation inequality, Lemma \ref{neg_sobo}. (i), with
\[
\frac{s}{d} = \lt( \frac12 - \frac{k}{d}\rt)\theta + \frac12(1-\theta), \quad k = \lt[\frac{d}{2}\rt] + 1-s, \quad \theta = \frac{\frac{d}{2}-s}{\lt[\frac{d}{2}\rt] +1- s}.
\]
For $\mathcal{I}_2$, we obtain
\[\begin{aligned}
\mathcal{I}_2&\le \|\Lambda^{-s-\frac{d-\alpha}{2}} ( h  u)\|_{L^2} \|\Lambda^{-s-\frac{d-\alpha}{2}} \nabla h\|_{L^2}\\
&\le \| h  u\|_{L^{\frac{1}{\frac{1}{2} + \frac{s+\frac{d-\alpha}{2}}{d}}}} \|\Lambda^{-s-\frac{d-\alpha}{2}} \nabla h\|_{L^2}\\
&\le \| h\|_{L^2} \|u\|_{L^{\frac{d}{s+\frac{d-\alpha}{2}}}}  \|\Lambda^{-s-\frac{d-\alpha}{2}} \nabla h\|_{L^2}\\
&\le  C\| h\|_{L^2}\|\nabla^{\lt[\frac{\alpha}{2} \rt]+1-s} u\|_{L^2}^{\theta} \| u\|_{L^2}^{1-\theta}  \|\Lambda^{-s-\frac{d-\alpha}{2}}\nabla h\|_{L^2}\\
&\le C\|h\|_{L^2}\|u\|_{H^{m+\frac{d-\alpha}{2}}}  \|\Lambda^{-s-\frac{d-\alpha}{2}} \nabla h\|_{L^2}.
\end{aligned}\]
Here we used Lemma \ref{neg_sobo}. (i) with
\[
\frac{s+\frac{d-\alpha}{2}}{d}= \lt( \frac12 - \frac{k}{d}\rt)\theta + \frac12(1-\theta), \quad k = \lt[\frac{\alpha}{2}\rt] + 1-s, \quad \theta = \frac{\frac{\alpha}{2}-s}{\lt[\frac{\alpha}{2}\rt] +1- s}.
\]
Now we combine all the estimates for $\mathcal{I}_i$'s to deduce the desired result.
\end{proof}

\begin{lemma}\label{neg_est2}
Let $T>0$, $m > \frac d2+2$, and $ 0< s\le \frac{\alpha}{2}$. Let $(h,u) \in \mc([0,T); X^m)$  be a solution to the system \eqref{np_ER2} satisfying
\[
\sup_{0\le t \le T} \|h(t)\|_{L^\infty} \le \frac12.
\]
 Then we have
\[\begin{aligned}
\frac{d}{dt}&\int_{\R^d} \Lambda^{-s}\nabla h \cdot \Lambda^{-s} u\,dx +\frac12\|\Lambda^{-s-\frac{d-\alpha}{2}}\nabla h\|_{L^2}^2\\
&\le C\lt(\|h\|_{H^m} + \|\Lambda^{-s-\frac{d-\alpha}{2}}\nabla h\|_{L^2}\rt) \|u\|_{H^{m+\frac{d-\alpha}{2}}}^2 + C\|h\|_{L^2}\|\Lambda^{-s}u\|_{L^2}^2\cr
&\quad + \|\Lambda^{-s}(\nabla \cdot u)\|_{L^2}^2 + \frac12\|\Lambda^{-s+\frac{d-\alpha}{2}} u\|_{L^2}^2,
\end{aligned}\] 
where $C > 0$ is independent of $T$.
\end{lemma}
\begin{proof}
Straightforward calculation yields
\[\begin{aligned}
\frac{d}{dt}\int_{\R^d} \Lambda^{-s}\nabla h \cdot \Lambda^{-s} u\,dx&= -\int_{\R^d} \Lambda^{-s} (\pa_t h) \Lambda^{-s} \nabla \cdot u\,dx + \int_{\R^d} \Lambda^{-s}\nabla  h \cdot \Lambda^{-s}  (\pa_t u)\,dx\\
&=: \mathcal{J}_1 + \mathcal{J}_2.
\end{aligned}\]
For $\mathcal{J}_1$, we use Lemma \ref{neg_sobo} to get
\[\begin{aligned}
\mathcal{J}_1&= \int_{\R^d} \Lambda^{-s}(\nabla \cdot (hu)) \Lambda^{-s} \nabla \cdot u\,dx  + \|\Lambda^{-s}(\nabla \cdot u)\|_{L^2}^2\\
&= -\int_{\R^d} \Lambda^{-s}(hu) \cdot \nabla \Lambda^{-s} \nabla \cdot u\,dx  + \|\Lambda^{-s}(\nabla \cdot u)\|_{L^2}^2\\
&\le \|\Lambda^{-s}(hu)\|_{L^2} \|\Lambda^{-s} \nabla (\nabla\cdot u)\|_{L^2} + \|\Lambda^{-s}(\nabla \cdot u)\|_{L^2}^2\\
&\le C \|hu\|_{L^{\frac{1}{\frac12 + \frac{s}{d}}}}\|\Lambda^{-s} \nabla (\nabla\cdot u)\|_{L^2} + \|\Lambda^{-s}(\nabla \cdot u)\|_{L^2}^2\\
&\le C\|h\|_{L^2} \|u\|_{H^{m+\frac{d-\alpha}{2}}}\|\Lambda^{-s} \nabla (\nabla\cdot u)\|_{L^2} + \|\Lambda^{-s}(\nabla \cdot u)\|_{L^2}^2\\
&\le C \|h\|_{L^2} \|u\|_{H^{m+\frac{d-\alpha}{2}}}\lt( \|u\|_{H^{m+\frac{d-\alpha}{2}}} + \|\Lambda^{-s}u\|_{L^2}\rt)+ \|\Lambda^{-s}(\nabla \cdot u)\|_{L^2}^2.
\end{aligned}\]
Here we used
\[
\|\Lambda^{-s}\nabla (\nabla \cdot u)\|_{L^2} \le \|\Lambda^{-s}u\|_{L^2}^{\frac{s}{s+2}} \|\nabla^2 u\|_{L^2}^{\frac{2}{s+2}} \le \frac{s}{s+2}\|\Lambda^{-s}u\|_{L^2} +  \frac{2}{s+2}\|\nabla^2 u\|_{L^2}.
\]
For $\mathcal{J}_2$, we also have
\[\begin{aligned}
\mathcal{J}_2&= -\int_{\R^d} \Lambda^{-s} \nabla h \cdot \Lambda^{-s}(u\cdot \nabla u)\,dx -\int_{\R^d} \Lambda^{-s} \nabla h \cdot \Lambda^{-s} u\,dx - \|\Lambda^{-s-\frac{d-\alpha}{2}} \nabla h\|_{L^2}^2\\
&\le \|\Lambda^{-s} \nabla h\|_{L^2}\|\Lambda^{-s}(u\cdot \nabla u)\|_{L^2} -\frac12 \|\Lambda^{-s-\frac{d-\alpha}{2}} \nabla h\|_{L^2}^2 + \frac12\|\Lambda^{-s+\frac{d-\alpha}{2}}u\|_{L^2}^2\\
&\le C\lt(\|\Lambda^{-s-\frac{d-\alpha}{2}} \nabla h\|_{L^2} + \|\nabla h\|_{L^2}\rt) \|u\|_{H^{m+\frac{d-\alpha}{2}}}^2 -\frac12 \|\Lambda^{-s-\frac{d-\alpha}{2}} \nabla h\|_{L^2}^2 + \frac12\|\Lambda^{-s+\frac{d-\alpha}{2}}u\|_{L^2}^2,
\end{aligned}\]
where we used Lemma \ref{neg_sobo} and Young's inequality to get
\[
\|\Lambda^{-s} \nabla h\|_{L^2} \le \|\Lambda^{-s-\frac{d-\alpha}{2}}\nabla h\|_{L^2}^\theta \|\nabla h\|_{L^2}^{1-\theta} \le \theta\|\Lambda^{-s-\frac{d-\alpha}{2}}\nabla h\|_{L^2} +  (1-\theta)\|\nabla h\|_{L^2}, \quad \theta := \frac{s}{s+\frac{d-\alpha}{2}}.
\]
Thus, we gather the estimates for $\mathcal{J}_1$ and $\mathcal{J}_2$ to conclude the desired result.
\end{proof}

\begin{proposition}\label{C2.3}
Let $T>0$, $m > \frac d2+2$, and $ 0< s\le \frac{\alpha}{2}$. Let $(h,u) \in \mc([0,T); X^m)$  be a solution to the system \eqref{np_ER2}.
Suppose that $X(T;m)\le \e_0^2 \ll 1$ so that
\[
\sup_{0\le t \le T} \|h(t)\|_{L^\infty} \le \frac12.
\]
 Then we have
\[\begin{aligned}
&\|u(\cdot,t)\|_{H^{m+\frac{d-\alpha}{2}}}^2  +  \|u(\cdot,t)\|_{\dot{H}^{-s}}^2 + \|h(\cdot,t)\|_{H^m}^2 + \|h(\cdot,t)\|_{\dot{H}^{-s-\frac{d-\alpha}{2}}}^2  \\
&\quad + \int_0^t \lt(\|u(\cdot,\tau)\|_{H^{m+\frac{d-\alpha}{2}}}^2 + \|u(\cdot,\tau)\|_{\dot{H}^{-s}}^2 +  \|\nabla h(\cdot,\tau)\|_{H^{m-1}}^2 + \|h(\cdot,\tau)\|_{\dot{H}^{1-s-\frac{d-\alpha}{2}}}^2\rt)\,d\tau\\
&\qquad \le C\lt(\|u_0\|_{H^{m+\frac{d-\alpha}{2}}}^2+ \|u_0\|_{\dot{H}^{-s}}^2 +  \|h_0\|_{H^m}^2+  \|h_0\|_{\dot{H}^{-s-\frac{d-\alpha}{2}}}^2\rt),
\end{aligned}\]
where $C > 0$ is independent of $T$.
\end{proposition}
\begin{proof}
We collect the estimates in Lemmas \ref{neg_est} and \ref{neg_est2}, combine this with \eqref{est_energy} and use Young's inequality to find positive constants $\lambda_2, \eta_3$ and $C_3$ satisfying
\[ 
\begin{aligned}
\frac{d}{dt}&\lt( \mathcal{Y}^m +\|\Lambda^{-s}u\|_{L^2}^2 + \|h\|_{L^2}^2 + \|h\|_{\dot{H}^{-\frac{d-\alpha}{2}}}^2 + \|\Lambda^{-s-\frac{d-\alpha}{2}} h\|_{L^2}^2 +\eta_3 \int_{\R^d} \Lambda^{-s}\nabla h \cdot \Lambda^{-s}u\,dx\rt) \\
&\qquad + \lambda_2 \lt( \mathcal{Y}^m + \|\Lambda^{-s}u\|_{L^2}^2 + \|\Lambda^{-s-\frac{d-\alpha}{2}}\nabla h\|_{L^2}^2\rt)\\
&\le C_3 \e_0\lt( \mathcal{Y}^m + \|\Lambda^{-s}u\|_{L^2}^2 + \|\Lambda^{-s-\frac{d-\alpha}{2}}\nabla h\|_{L^2}^2\rt),
\end{aligned}
\] 
where we used
\[
\|\Lambda^{-s}(\nabla \cdot u)\|_{L^2}\lesssim_{s} \|\Lambda^{-s}u \|_{L^2}^{\frac{s}{s+1}} \|\nabla u\|_{L^2}^{\frac{1}{s+1}}
\]
and
\[
\|\Lambda^{-s+\frac{d-\alpha}{2}}u\|_{L^2}\lesssim_{s,d,\alpha} \|\Lambda^{-s}u \|_{L^2}^{\frac{s-\frac{d-\alpha}{2}}{s}} \|u\|_{L^2}^{\frac{\frac{d-\alpha}{2}}{s}}.
\]
Thus, we use the smallness of $\e_0$ to get a constant $\lambda_3>0$ and the relation
\[
\|\Lambda^{-\frac{d-\alpha}{2}} h\|_{L^2} \le \|\Lambda^{-s-\frac{d-\alpha}{2}} h\|_{L^2}^{\theta} \|h\|_{L^2}^{1-\theta} \le \theta \|\Lambda^{-s-\frac{d-\alpha}{2}} h\|_{L^2}  + (1-\theta)\|h\|_{L^2}, \quad \theta:= \frac{\frac{d-\alpha}{2}}{s+\frac{d-\alpha}{2}}
\]
and
\[
\|\Lambda^{-s} \nabla h\|_{L^2} \le \|\Lambda^{-s-\frac{d-\alpha}{2}} h\|_{L^2}^{\bar\theta} \|\nabla h\|_{L^2}^{1-\bar\theta} \le \bar\theta \|\Lambda^{-s-\frac{d-\alpha}{2}} h\|_{L^2}  + (1-\bar\theta)\|h\|_{L^2}, \quad \bar\theta:= \frac{s}{1 + s+\frac{d-\alpha}{2}}
\]
to yield
\[\begin{aligned}
&\mathcal{Y}^m(t) + \|u(\cdot,t)\|_{\dot{H}^{-s}}^2 + \|h(\cdot,t)\|_{L^2}^2+  \|h(\cdot,t)\|_{\dot{H}^{-s-\frac{d-\alpha}{2}}}^2 \cr
&\quad + \lambda_3 \int_0^t \lt( \mathcal{Y}^m(\tau)  + \|u(\cdot,\tau)\|_{\dot{H}^{-s}}^2 + \|h(\cdot,\tau)\|_{\dot{H}^{1-s-\frac{d-\alpha}{2}}}^2\rt) \,d\tau \\
&\qquad \le \mathcal{Y}^m(0) + \|u_0\|_{\dot{H}^{-s}}^2 +  \|h_0\|_{L^2}^2 + \|h_0\|_{\dot{H}^{-s-\frac{d-\alpha}{2}}}^2.
\end{aligned}\]
This asserts the desired result. 
\end{proof}

\subsubsection{Proof of Theorem \ref{main_thm2}} In this part, we provide the details of proof of Theorem \ref{main_thm2} on the large-time behavior of solutions in the whole space. 

Before getting into the main estimates, we first deal with the decay estimate \eqref{decay_weak} in Remark \ref{rmk_decay}. That introduces the main ideas behind our arguments for the better decay estimates of solutions. 

From Lemma \ref{neg_sobo}, we have that for $s \ge 0$,
\[
\|h\|_{L^2} \le \|\nabla h\|_{L^2}^{\frac{s+\frac{d-\alpha}{2}}{1+s+\frac{d-\alpha}{2}}} \|h\|_{\dot{H}^{-s-\frac{d-\alpha}{2}}}^{\frac{1}{1+s+\frac{d-\alpha}{2}}} \quad \mbox{and} \quad \|h\|_{\dot{H}^{-\frac{d-\alpha}{2}}}\le \|\nabla h\|_{L^2}^{\frac{s}{1+s+\frac{d-\alpha}{2}}} \|h\|_{\dot{H}^{-s-\frac{d-\alpha}{2}}}^{\frac{1+\frac{d-\alpha}{2}}{1+s+\frac{d-\alpha}{2}}}.
\] 
We then use the uniform bound in Proposition \ref{C2.3} and the smallness assumptions on the solutions to get
\[
\lt(\|h\|_{L^2}^2 + \|h\|_{\dot{H}^{-\frac{d-\alpha}{2}}}^2 \rt)^{\frac{1+s+\frac{d-\alpha}{2}}{s}}\lesssim \|\nabla h\|_{L^2}^2. 
\]
Thus, we now set
\bq\label{def_f}
\mathcal{F}^m(t) := \mathcal{Y}^m(t) + \|h(\cdot,t)\|_{L^2}^2+ \eta_2\|h(\cdot,t)\|_{\dot{H}^{-\frac{d-\alpha}{2}}}^2,
\eq
where $\mathcal{Y}^m$ is from Proposition \ref{C2.3}. From the smallness of solutions and estimates in \eqref{est_energy}, we can find a constant $\lambda_4>0$ satisfying
\[
\frac{d}{dt}\mathcal{F}^m + \lambda_4(\mathcal{F}^{m})^{\frac{1+s+\frac{d-\alpha}{2}}{s}} \le 0.
\]
This gives
\[
-\frac{1}{\frac{1+\frac{d-\alpha}{2}}{s}}\frac{d}{dt} \lt( (\mathcal{F}^m)^{-\frac{1+\frac{d-\alpha}{2}}{s}}\rt) \le -\lambda_4.
\]
We integrate the above with respect to $t$ and get
\[
(\mathcal{F}^m)^{-\frac{1+\frac{d-\alpha}{2}}{s}} \ge (\mathcal{F}^m(0))^{-\frac{1+\frac{d-\alpha}{2}}{s}} + \lambda_4 \frac{1+\frac{d-\alpha}{2}}{s} t,
\]
or equivalently,
\[
\mathcal{F}^m(t) \le \lt((\mathcal{F}^m(0))^{-\frac{1+\frac{d-\alpha}{2}}{s}} + \lambda_4 \frac{1+\frac{d-\alpha}{2}}{s} t\rt)^{-\frac{s}{1+\frac{d-\alpha}{2}}}.
\]
Thus, we obtain
\[
\mathcal{F}^m(t) \lesssim (1+t)^{-\frac{s}{1+\frac{d-\alpha}{2}}}.
\]
Since $\mathcal{F}^m(t) \approx \|(h,u)(\cdot,t)\|_{X^m}^2$, this concludes the desired result.

As mentioned in Remark \ref{rmk_decay}, the above estimates do not allow us to have a better decay rate of convergence even in higher dimensions. For that reason, we refine the above arguments by taking the negative order $(\frac\alpha2 \geq ) s > 0$ large enough. 

We present the proof by dividing into two cases; $s \geq 1 - \frac{d-\alpha}{2}$ and  
$s>2+d-\alpha$. In the former case, we obtain the algebraic decay rate of convergence of solutions. On the other hand, in the latter case, the exponential decay rate is found.
\newline

\noindent $\diamond$ (Case A: $s \geq 1 - \frac{d-\alpha}{2}$) In this case,  we first note that 
\[
-s -\frac{d-\alpha}{2} \leq -1 + \frac{d-\alpha}{2} < 0 \quad \mbox{and} \quad - s < d - \alpha - 1 < 1
\]
provided our assumption on $s$ and $\alpha$, and thus the interpolation inequality implies
\[
(h_0, u_0) \in \dot{H}^{-1+\frac{d-\alpha}{2}} (\R^d) \times [\dot{H}^{d-\alpha-1}(\R^d)]^d.
\]
Then, similarly to Lemma \ref{neg_est}, we estimate
\begin{align*}
\frac12&\frac{d}{dt}\lt( \|\Lambda^{-1+\frac{d-\alpha}{2}} h\|_{L^2}^2 + \|\Lambda^{d-\alpha-1} u\|_{L^2}^2\rt) + \|\Lambda^{d-\alpha-1} u\|_{L^2}^2\\
&= -\int_{\R^d} (\Lambda^{-1+\frac{d-\alpha}{2}} h )\Lambda^{-1+\frac{d-\alpha}{2}} (\nabla \cdot (h u))\,dx - \int_{\R^d} (\Lambda^{-1+\frac{d-\alpha}{2}} h) \Lambda^{-1+\frac{d-\alpha}{2}} (\nabla \cdot  u)\,dx\\
&\quad - \frac12\int_{\R^d} \Lambda^{d-\alpha-1} (u\cdot \nabla u) \Lambda^{d-\alpha-1} u\,dx -\int_{\R^d} \nabla\Lambda^{-1}h \Lambda^{d-\alpha-1} u\,dx\\
&= \int_{\R^d} (\Lambda^{-1+\frac{d-\alpha}{2}}\nabla h ) \cdot \Lambda^{-1+\frac{d-\alpha}{2}} (hu)\,dx - \frac12\int_{\R^d} \Lambda^{d-\alpha-1} (u\cdot \nabla u) \Lambda^{d-\alpha-1} u\,dx\\
&\le \|\Lambda^{-1+\frac{d-\alpha}{2}}\nabla h\|_{L^2}\|\Lambda^{-1+\frac{d-\alpha}{2}} (h  u) \|_{L^2}  +\frac12\|\Lambda^{d-\alpha-1}(u \cdot \nabla u)\|_{L^2}\|\Lambda^{d-\alpha-1}u\|_{L^2}\\
&\le \|\Lambda^{-1+\frac{d-\alpha}{2}} \nabla h\|_{L^2}\|h  u\|_{L^{\frac{1}{\frac12 + \frac{1-\frac{d-\alpha}{2}}{d}}}} +\frac12\|\Lambda^{d-\alpha-1}(u \cdot \nabla u)\|_{L^2}\|\Lambda^{d-\alpha-1}u\|_{L^2}.
\end{align*}
Here, if $\alpha \le d-1$, we get
\[
\|\Lambda^{d-\alpha-1}(u\cdot \nabla u)\|_{L^2} \le C\|u\cdot \nabla u\|_{H^1} \le C \|u\|_{H^{m+\frac{d-\alpha}{2}}}^2,
\]
and if $\alpha\in (d-1,d)$, we deduce
\[
\|\Lambda^{-(d-\alpha-1)} (u\cdot \nabla u)\|_{L^2} \le \|u \cdot \nabla u\|_{L^{\frac{1}{\frac12 + \frac{d-\alpha-1}{d}}}} \le \|\nabla u\|_{L^2} \|u\|_{L^{\frac{d}{d-\alpha-1}}} \le C\|u\|_{H^{m+\frac{d-\alpha}{2}}}^2.
\]
In either case, we have
\begin{align*}
\frac12&\frac{d}{dt}\lt( \|\Lambda^{-1+\frac{d-\alpha}{2}} h\|_{L^2}^2 + \|\Lambda^{d-\alpha-1} u\|_{L^2}^2\rt) + \|\Lambda^{d-\alpha-1} u\|_{L^2}^2\\
&\le \|\Lambda^{-1+\frac{d-\alpha}{2}} \nabla h\|_{L^2}\|h  u\|_{L^{\frac{1}{\frac12 + \frac{1-\frac{d-\alpha}{2}}{d}}}} +C\|u\|_{H^{m+\frac{d-\alpha}{2} }}^2\|\Lambda^{d-\alpha-1}u\|_{L^2}\\
&\le \|\Lambda^{-1+\frac{d-\alpha}{2}} \nabla h\|_{L^2}\|h\|_{L^2}\|  u\|_{L^{\frac{d}{1-\frac{d-\alpha}{2}}}} +C\|u\|_{H^{m+\frac{d-\alpha}{2}}}^2\|\Lambda^{d-\alpha-1}u\|_{L^2}\\
&\le C\|h\|_{H^m}^2 \|u\|_{H^{m+\frac{d-\alpha}{2}}}+C\|u\|_{H^{m+\frac{d-\alpha}{2}}}^2\|\Lambda^{d-\alpha-1}u\|_{L^2},
\end{align*}
where we used $1-\frac{d-\alpha}{2} \in (0,1)$ to have $\|\Lambda^{-1+\frac{d-\alpha}{2}} \nabla h\|_{L^2} \lesssim \|h\|_{H^m}$. On the other hand, similarly to Lemma \ref{neg_est2}, we can get a constant $\gamma>0$ satisfying
\begin{align*}
-\frac{d}{dt}&\int_{\R^d} h\Lambda^{d-\alpha -2}\nabla \cdot u\,dx\\
&= -\int_{\R^d} (\pa_t h)\Lambda^{d-\alpha-2}\nabla \cdot u \,dx - \int_{\R^d} h \Lambda^{-\alpha }(\pa_t (\nabla \cdot u))\,dx\\
&=-\int_{\R^d} hu\cdot \Lambda^{d-\alpha-2}  \nabla(\nabla \cdot u)\,dx + \|\Lambda^{-1+\frac{d-\alpha}{2}} \nabla \cdot u\|_{L^2}^2\\
&\quad +\frac12\int_{\R^d} h \Lambda^{d-\alpha-2} (\nabla \cdot (u \cdot \nabla u))\,dx - \int_{\R^d} h \Lambda^{d-\alpha-2} \nabla \cdot u\,dx - \|h\|_{L^2}^2\\
&\le  \|u\|_{L^\infty} \|h\|_{L^2} \|\Lambda^{d-\alpha-2} \nabla (\nabla \cdot u)\|_{L^2} + \|\Lambda^{-1+\frac{d-\alpha}{2}}\nabla \cdot u\|_{L^2}^2  +C\|h\|_{L^2}\|\Lambda^{d-\alpha-2}\nabla\cdot (u\cdot \nabla u)\|_{L^2}\\
&\quad- \frac12\|h\|_{L^2}^2  + \frac12\|\Lambda^{d-\alpha-2} \nabla \cdot u\|_{L^2}^2\\
&\le C\|h\|_{L^2}\|u\|_{H^{m+\frac{d-\alpha}{2}}}^2 + \gamma\lt( \|u\|_{H^1}^2 + \|\Lambda^{d-\alpha-1} u\|_{L^2}^2\rt) - \frac12\|h\|_{L^2}^2.
\end{align*}
Then we use the estimates in Proposition \ref{C2.3} to get
\bq\label{case2-1}
\begin{aligned}
\frac{d}{dt}&\lt( \mathcal{F}^m + \|\Lambda^{-1+\frac{d-\alpha}{2}} h\|_{L^2}^2 + \|\Lambda^{d-\alpha-1} u\|_{L^2}^2 - \eta_4\int_{\R^d} h\Lambda^{d-\alpha-2} \nabla \cdot u\,dx\rt) \\
&\qquad+ \eta_5 \lt(\mathcal{Y}^m + \|\Lambda^{d-\alpha-1} u\|_{L^2}^2 + \|h\|_{L^2}^2\rt)\\
&\le C\e_0 \lt( \mathcal{Y}^m + \|\Lambda^{d-\alpha-1}u\|_{L^2}^2 + \|h\|_{L^2}^2\rt),
\end{aligned}
\eq
for some positive constants $\eta_4$ and $\eta_5$, where $\mathcal{F}^m$ is appeared in \eqref{def_f}. Noting that
\[
s+\frac{d-\alpha}{2} \ge 1> \max\lt\{1-\frac{d-\alpha}{2}, \frac{d-\alpha}{2}\rt\},
\] 
we find
\[
\|h\|_{\dot{H}^{-\frac{d-\alpha}2}}\le \|h\|_{L^2}^{\frac{s}{s+\frac{d-\alpha}{2}}} \|h\|_{\dot{H}^{-s-\frac{d-\alpha}{2}}}^{\frac{\frac{d-\alpha}{2}}{s+\frac{d-\alpha}{2}}} \quad \mbox{and} \quad \|h\|_{\dot{H}^{-1+\frac{d-\alpha}{2}}} \le \|h\|_{L^2}^{\frac{s+d-\alpha-1}{s+\frac{d-\alpha}{2}}} \|h\|_{\dot{H}^{-s-\frac{d-\alpha}{2}}}^{\frac{1-\frac{d-\alpha}{2}}{s+\frac{d-\alpha}{2}}},
\]
which subsequently imply
\[
\|h\|_{\dot{H}^{-\frac{d-\alpha}{2}}}^{\frac{s+\frac{d-\alpha}{2}}{s}} \lesssim \|h\|_{L^2} \quad \mbox{and} \quad \|h\|_{\dot{H}^{-1+\frac{d-\alpha}{2}}}^{\frac{s+\frac{d-\alpha}{2}}{s+d-\alpha-1}} \lesssim \|h\|_{L^2}.
\]
From the smallness condition on $\|h\|_{L^2}$, we have
\bq\label{h_small}
(\|h\|_{\dot{H}^{-1+\frac{d-\alpha}{2}}}^2 +  \|h\|_{\dot{H}^{-\frac{d-\alpha}2}}^2)^{1+\zeta} \lesssim \|h\|_{L^2}^2, \quad \zeta:= \max\lt\{ \frac{d-\alpha}{2s}, \frac{1-\frac{d-\alpha}{2}}{s+d-\alpha-1}\rt\}.
\eq
We now define
\[
\mathcal{Z}^m := \mathcal{F}^m + \|\Lambda^{-1+\frac{d-\alpha}{2}} h\|_{L^2}^2 + \|\Lambda^{d-\alpha-1} u\|_{L^2}^2 - \eta_4\int_{\R^d} h\Lambda^{d-\alpha-2} \nabla \cdot u\,dx,
\]
then a simple combination \eqref{case2-1} and \eqref{h_small} leads to
\[
\frac{d}{dt}\mathcal{Z}^m + \lambda_3 (\mathcal{Z}^m)^{1+\zeta} \le 0.
\]
Solving the above differential inequality gives
\[
(\mathcal{Z}^m)(t) \le  \lt((\mathcal{Z}^m(0))^{-\zeta} + \lambda_3 \zeta t\rt)^{-\frac1\zeta},
\]
and this proves the first assertion in Theorem \ref{main_thm2}. \newline

\noindent $\diamond$ (Case B: $s>2+d-\alpha$) Note that
\[
\|u\|_{L^2} \le \|u\|_{\dot{H}^{1+\frac{d-\alpha}{2}}}^{\frac{s}{1+\frac{d-\alpha}{2} +s}}\|u\|_{\dot{H}^{-s}}^{\frac{1+\frac{d-\alpha}{2}}{1+\frac{d-\alpha}{2}+s}}
\]
and $s>2+d-\alpha$ is equivalent to $\frac{s}{1+\frac{d-\alpha}{2}+s} >\frac23$. Since we have the uniform bound for $\|u\|_{\dot{H}^{-s}}$, we can get
\[
\|u\|_{L^2}^3 \lesssim \e_0^{\frac{3s}{1+\frac{d-\alpha}{2}+s}-2} \|u\|_{\dot{H}^{1+\frac{d-\alpha}{2}}}^2 =\e_0^{\frac{3s}{1+\frac{d-\alpha}{2}+s}-2} \|U_1\|_{L^2}^2,
\]
where $U_1$ was defined as $U_1 := \nabla \Lambda^{\frac{d-\alpha}{2}} u$. 

We now define a function
\bq\label{def_y2}
\overline{\mathcal{Y}}^m:= \sum_{1 \le k \le m} \lt(\|U_k\|_{L^2}^2 +\lt\|\frac{1}{\sqrt{\rho}} \pa^k h\rt\|_{L^2}^2\rt) +\eta_1 \sum_{0 \le k \le m-1} \int_{\R^d} \frac1\rho \pa^k \nabla h \cdot \Lambda^{d-\alpha} \pa^k u\,dx,
\eq
then it follows from \eqref{est_he} that there exist positive constants $\eta_5$ and $C$, independent of $\e_0$ and $T$, such that 
\[
\begin{aligned}
\frac{d}{dt} \overline{\mathcal{Y}}^m + \eta_5 \overline{\mathcal{Y}}^m &\le C\e_0 \overline{\mathcal{Y}}^m + C\|u\|_{L^2}^3 \le  C\e_0 \overline{\mathcal{Y}}^m + C \e_0^{\frac{3s}{1+\frac{d-\alpha}{2}+s}-2}\|U_1\|_{L^2}^2.
\end{aligned}
\] 
Then, the smallness condition on $\e_0$ implies
\[
\frac{d}{dt} \overline{\mathcal{Y}}^m + \eta_7\overline{\mathcal{Y}}^m \le 0,
\]
for some constant $\eta_7>0$ and Gr\"onwall's lemma implies the exponential decay rate of convergence of solutions. This completes the proof.

%
%
%
%
%
\subsection{Periodic case: exponential decay rate of convergence}
\setcounter{equation}{0}
In this part, we consider the periodic domain, i.e. $\Omega = \T^d$, and study the large-time behavior estimates of system \eqref{np_ER2}. Instead of dealing with the negative Sobolev norm of $(h,u)$, we take the advantage of the boundedness of the domain and show the exponential decay estimate of its $L^2$ norm. 

Let us define a modulated energy:
\[
\me(t) := \frac12 \int_{\T^d} (1+h) |u-m_c|^2\,dx +\frac12 \int_{\T^d} h \Lambda^{\alpha-d} h\,dx,
\]
where $m_c$ denotes the average of the momentum:
\[
m_c := \int_{\T^d} (1+h) u\,dx.  
\]
Here we remind the reader that the mass is assumed to be zero, i.e. $\int_{\T^d} h\,dx = 0$ for all $t \geq 0$. 

Note that if there exists a positive constant lower bound on $1+h$, i.e. $h(x,t) +1 > h_{min} > 0$ for all $(x,t) \in \T^d \times \R_+$, then the modulated energy satisfies
\bq\label{est_me0}
h_{min} \|(u - m_c)(\cdot,t)\|_{L^2}^2 + \|h(\cdot,t)\|_{\dot{H}^{-\frac{d-\alpha}{2}}}^2 \leq \me(t),
\eq
for all $t \geq 0$. This implies that the exponential decay of $\me(t)$ also gives the estimate of the lowest order norm of solutions. For that reason, our first goal is to prove the following proposition.
\begin{proposition}\label{main_prop} Let $(h, u)$ be a global classical solution to \eqref{np_ER2} with sufficient regularity. Suppose that
\begin{enumerate}
\item[(i)]
$\inf\limits_{(x,t)\in \T^d \times \R_+} 1 + h(x,t) \ge h_{min}>0$ and
\item[(ii)]
$h \in W^{1,\infty}(\T^d \times \R_+)$, $\nabla u \in L^\infty(\R_+;[L^{\infty}(\T^d)]^d)$.
\end{enumerate}
Then we have
\[
\|(u - m_c)(\cdot,t)\|_{L^2}^2+\|h(\cdot,t)\|_{\dot{H}^{-\frac{d-\alpha}{2}}} \le Ce^{-\lambda t}.
\]
Here $C$ and $\lambda$ are positive constants independent of $t$.
\end{proposition}
\begin{remark}We notice that the required regularity and assumptions for solutions $(h,u)$ are guaranteed by Theorem \ref{main_thm}.
\end{remark}

In the lemma below, we first show that the modulated energy $\me(t)$ is not increasing in time.
\begin{lemma}\label{P3.1}
Let $(h, u)$ be a global classical solution to \eqref{np_ER2} with sufficient regularity. Then we have
\[
\frac{d}{dt}\me(t) + \md(t) = 0,
\]
where the dissipation rate function $\md$ is given by
\[
\md(t) := \int_{\T^d} (h+1)|u-m_c|^2\,dx.
\]
\end{lemma}
\begin{proof}
Direct computation gives
\[\begin{aligned}
&\frac12\frac{d}{dt}\int_{\T^d} (h+1)|u-m_c|^2\,dx \cr
&\quad = \frac12 \int_{\T^d} \pa_t h |u-m_c|^2\,dx + \int_{\T^d} (h+1)(u-m_c) (\pa_t u - m_c')\,dx\\
&\quad = \int_{\T^d} ((h+1) u \cdot \nabla u)\cdot (u-m_c)\,dx -\int_{\T^d} (h+1)(u-m_c)\cdot \lt( u \cdot \nabla u + u +\nabla \Lambda^{\alpha-d} h\rt)\,dx\\
&\quad = -\int_{\T^d} (h+1) (u-m_c)\cdot u\,dx -\int_{\T^d} (h+1)(u-m_c) \cdot \nabla \Lambda^{\alpha-d}h\,dx\\
&\quad = -\int_{\T^d} (h+1) |u-m_c|^2\,dx - \int_{\T^d} (h+1) u \cdot \nabla \Lambda^{\alpha-d}h\,dx.
\end{aligned}\]
Here we used the symmetry of the operator $\Lambda^{\alpha-d}$:
\[
\int_{\T^d} (h+1)  \nabla \Lambda^{\alpha-d}h \,dx = \int_{\T^d} h \nabla \Lambda^{\alpha-d}h\,dx = 0.
\]
Since we have
\[
\frac{1}{2}\frac{d}{dt}\int_{\T^d} h\Lambda^{\alpha-d}h\,dx = \int_{\T^d} (h+1)u \cdot \nabla \Lambda^{\alpha-d}h\,dx,
\]
we combine the above results to conclude the desired result.
\end{proof}
Since the dissipation rate function $\md$ does not have a dissipation with respect to $h$, motivated from \cite{CJ21}, we introduce a perturbed modulated energy $\me^\sigma$:
\[
\me^\sigma := \me +\sigma\int_{\T^d} (u-m_c) \cdot \nabla W \star h\,dx,
\]
where $\sigma > 0$ will be chosen appropriately later and $W$ satisfies the following relations:
\begin{enumerate}
\item[(i)] The potential $W$ is an even function explicitly written as
\[
W(x) = \left\{\begin{array}{lcl}\displaystyle  -c_0 \log |x| + G_0(x) & \mbox{ if } & d=2, \\
c_1|x|^{2-d} + G_1(x) & \mbox{ if } & d\ge 3,\end{array}\right.
\]
where $c_0>0$ and $c_1>0$ are normalization constants and $G_0$ and $G_1$ are smooth functions over $\T^2$ and $\T^d$ ($d\ge 3$), respectively. 
\vspace{0.2cm}

\item[(ii)] For any $h \in L^2(\T^d)$ with $\int_{\T^d} h \,dx =0$, $U := W\star h \in H^1(\T^d)$ is the unique function that satisfies the following condition:
\bq\label{prop_W}
\int_{\T^d} U\,dx = 0 \quad \mbox{and} \quad \int_{\T^d} \nabla U \cdot \nabla \psi\,dx = \int_{\T^d} h\,\psi\,dx \quad \forall\, \psi \in H^1(\T^d),
\eq
i.e. $U$ is the unique weak solution to $-\Delta U = h$.
\end{enumerate}

\begin{remark}\label{R3.1}
For $h \in L^2(\T^d)$ with $\int_{\T^d} h\,dx = 0$, the followings hold:
\begin{enumerate}
\item[(i)]
For $W$ defined above, we have
\[
\|h\|_{\dot{H}^{-1}} \approx \|\nabla W \star h\|_{L^2}.
\]
\item[(ii)]
We can define $\dot{H}^{-s}(\T^d)$-norm as
\[
\|h\|_{\dot{H}^{-s}}:= \lt(\sum_{n \in \Z^d} |n|^{-2s} |\hat{h}(n)|^2 \rt)^{1/2}.
\]
Then, by definition, it is clear that  for $s_1 \ge s_2 \ge 0$
\[
 \|h\|_{\dot{H}^{-s_1}} \le \|h\|_{\dot{H}^{-s_2}}.\]
 In particular, we have
 \[
 \|h\|_{\dot{H}^{-1}} \le \|h\|_{\dot{H}^{-\frac{d-\alpha}{2}}},
 \]
 due to $(d-\alpha)/2 \in (0,1)$.
\end{enumerate}
\end{remark}

\begin{remark}
Due to Remark \ref{R3.1} and the assumptions in Proposition \ref{main_prop}, we find
\[
\lt|\int_{\T^d} (u-m_c)\cdot \nabla W \star h\,dx\rt| \ls \int_{\T^d} (h+1)|u-m_c|^2\,dx + \|h\|_{\dot{H}^{-1}}.
\]
Since $\|h\|_{\dot{H}^{-1}} \le \|h\|_{\dot{H}^{-\frac{d-\alpha}{2}}}$, for sufficiently small $\sigma$, we get
\bq\label{est_me1}
\me \approx \me^\sigma.
\eq
\end{remark}

We are now in a position to provide the details on the proof for Proposition \ref{main_prop}.
\begin{proof}[Proof of Proposition \ref{main_prop}] It is obvious that $\me^\sigma$ satisfies
\bq\label{est_me2}
\frac{d}{dt}\me^\sigma + \md^\sigma = 0,
\eq
where $\md^\sigma = \md^\sigma(t)$ is given as
\[
\md^\sigma := \md - \sigma\frac{d}{dt}\int_{\T^d}(u-m_c)\cdot \nabla W \star h\,dx.
\]
Now, we claim that
\[
\md^\sigma(t) \ge c \me^\sigma(t),
\]
for some positive constant $c$ independent of $t$. First, we estimate
\begin{align*}
\frac{d}{dt}&\int_{\T^d} (u-m_c)\cdot \nabla W\star h \,dx\\
&= -\int_{\T^d} \lt(u\cdot \nabla u + u + \nabla \Lambda^{\alpha-d}h \rt)\cdot \nabla W \star h\,dx  -m_c' \int_{\T^d} \nabla W \star h \,dx + \int_{\T^d} (u-m_c) \cdot \nabla W\star(\pa_t h)\,dx\\
&= -\int_{\T^d} (u \cdot \nabla u)\cdot \nabla W \star h \,dx -\int_{\T^d} (u-m_c) \cdot \nabla W\star h \,dx\\
&\quad -\int_{\T^d} \nabla \Lambda^{\alpha-d}h \cdot \nabla W\star h \,dx -\int_{\T^d} (u-m_c)\cdot \nabla W\star(\nabla \cdot ((h+1) u))\,dx\\
&=: \sum_{i=1}^4 \mathcal{J}_i.
\end{align*}
For $\mathcal{J}_1$, we recall that for $a = (a_1,\dots,a_d) \in \R^d$ and $b = (b_1,\dots,b_d) \in \R^d$
\[
\nabla \cdot (a \otimes b) = \sum_{j=1}^d \pa_{x_j}( a_i b_j) = a (\nabla \cdot b) + (b\cdot\nabla)a.
\]
This gives
\[\begin{aligned}
\mathcal{J}_1 &= -\int_{\T^d} (u \cdot \nabla (u-m_c))\cdot \nabla W \star h \,dx\\
&= -\int_{\T^d} \nabla \cdot ((u-m_c) \otimes u)\cdot \nabla W\star h \,dx +\int_{\T^d} (u-m_c) (\nabla \cdot u) \cdot \nabla W\star h \,dx\\
&= \int_{\T^d} ((u-m_c) \otimes u): \nabla^2 W\star h \,dx +\int_{\T^d} (u-m_c) (\nabla \cdot u) \cdot \nabla W\star h \,dx.
\end{aligned}\]
For $\mathcal{J}_3$, we use \eqref{prop_W} to get
\[
\mathcal{J}_3 = \int_{\T^d} (\Lambda^{\alpha-d} h)  \Delta W\star h \,dx = -\int_{\T^d} h \Lambda^{\alpha-d} h \,dx.
\]
For $\mathcal{J}_4$, we find
\[
\mathcal{J}_4 = -\int_{\T^d} (u-m_c) \cdot \nabla W\star(\nabla \cdot ((h+1)(u-m_c)))\,dx -\int_{\T^d} (u-m_c) \cdot \nabla W \star(\nabla \cdot ((h+1) m_c))\,dx.
\]
Here we rewrite the second term on the right hand side of the above as
\begin{align*}
-\int_{\T^d} (u-m_c) \cdot \nabla W \star(\nabla \cdot((h+1) m_c))\,dx
&= \int_{\T^d} (\nabla \cdot (u-m_c))  W\star(\nabla \cdot ((h+1) m_c))\,dx\\
&= \iint_{\T^d \times \T^d} (\nabla \cdot (u-m_c)) W(x-y)\nabla_y \cdot (h(y)m_c)\,dxdy\\
&=\iint_{\T^d \times \T^d} (\nabla \cdot (u-m_c)) m_c \cdot \nabla W(x-y) h(y)\,dxdy\\
&= \int_{\T^d} m_c( \nabla \cdot (u-m_c))\cdot \nabla W\star h \,dx\\
&= \int_{\T^d} \nabla \cdot (m_c \otimes (u-m_c)) \cdot \nabla W\star h \,dx\cr
&= -\int_{\T^d} (m_c \otimes (u-m_c)) : \nabla^2 W\star h\,dx\cr
&=-\int_{\T^d} ((u-m_c)\otimes m_c) : \nabla^2 W\star h\,dx,
\end{align*}
where we used the symmetry of $\nabla^2 W\star h $ to get the last equality. 
Thus, $\mathcal{J}_4$ can be estimated as
\[\begin{aligned}
\mathcal{J}_4 
&=-\int_{\T^d} (u-m_c) \cdot \nabla W\star(\nabla \cdot ((h+1)(u-m_c)))\,dx -\int_{\T^d} ((u-m_c)\otimes m_c) : \nabla^2 W\star h \,dx.
\end{aligned}\]
Hence, we combine the estimates for $J_i$'s to yield
\begin{align*}
\frac{d}{dt}&\int_{\T^d} (u-m_c)\cdot \nabla W\star h \,dx\\
&= \int_{\T^d}( (u-m_c) \otimes (u-m_c) ): \nabla^2 W\star h \,dx + \int_{\T^d} (\nabla \cdot u  -1) (u-m_c)\cdot \nabla W\star h \,dx\\
&\quad -\int_{\T^d} h \Lambda^{\alpha-d}h\,dx -\int_{\T^d} (u-m_c)\cdot \nabla W\star(\nabla \cdot ((h+1)(u-m_c)))\,dx.
\end{align*}
Therefore, we choose sufficiently small $\sigma>0$ to obtain
\begin{align*}
\md^\sigma &= \md  - \sigma  \int_{\T^d}( (u-m_c) \otimes (u-m_c) ): \nabla^2 W\star h \,dx -\sigma \int_{\T^d} (\nabla \cdot u  -1) (u-m_c)\cdot \nabla W\star h \,dx\\
&\quad + \sigma \int_{\T^d} h \Lambda^{\alpha-d} h \,dx + \sigma\int_{\T^d} (u-m_c)\cdot \nabla W\star(\nabla \cdot ((h+1)(u-m_c)))\,dx\\
&\ge \md - \sigma \|\nabla^2 W\star h\|_{L^\infty} \|(u-m_c)\|_{L^2}^2 - \sigma(1+\|\nabla \cdot u\|_{L^\infty}) \|u-m_c\|_{L^2} \|\nabla W\star h\|_{L^2}\\
&\quad +\sigma \|h\|_{\dot{H}^{-\frac{d-\alpha}{2}}}^2 - \sigma \|u-m_c\|_{L^2} \|\nabla W\star(\nabla \cdot ((h+1)(u-m_c)))\|_{L^2}\\
&\ge \md -C\sigma \|\nabla W\|_{L^1}\|\nabla h\|_{L^\infty}\lt(\int_{\T^d} (h+1)|u-m_c|^2\,dx \rt) - C\sigma \lt(\int_{\T^d} (h+1) |u-m_c|^2\,dx\rt)^{1/2} \|h\|_{\dot{H}^{-1}}\\
&\quad  +\sigma \|h\|_{\dot{H}^{-\frac{d-\alpha}{2}}}^2 - C\sigma \|u-m_c\|_{L^2} \|\nabla \cdot ((h+1)(u-m_c)))\|_{\dot{H}^{-1}}\\
& \ge (1- C(\sigma^{1/2} + \sigma))\md  + (\sigma - C\sigma^{3/2})\|h\|_{\dot{H}^{-\frac{d-\alpha}{2}}}^2\\
&\ge c \me^\sigma,
\end{align*}
where $c = c(h_{min}, \|\nabla W\|_{L^1}, \|h\|_{W^{1,\infty}}, \|\nabla u\|_{L^\infty}) $ is a positive constant independent of $t$. Thus the claim is proved, and we have from \eqref{est_me2} that
\[
\frac{d}{dt} \me^\sigma + c \me^\sigma \leq 0.
\]
Applying Gr\"onwall's lemma to above deduces the exponential decay of $\me^\sigma$, and this, combined with \eqref{est_me0} and \eqref{est_me1}, concludes the desired result.
\end{proof}

\subsubsection{Proof of Theorem \ref{main_thm2}: periodic case} We first notice that the average of momentum satisfies
\[
m_c'(t) = - m_c(t), \quad \mbox{i.e.} \quad m_c(t) = m_c(0) e^{-t},
\]
due to the symmetry of the operator $\Lambda^{\alpha-d}$. This together with Proposition \ref{main_prop} gives
\[
\int_{\T^d} |u|^2\,dx \leq 2\int_{\T^d} |u - m_c|^2\,dx + 2|m_c|^2 \leq C_1e^{-C_2t},
\] 
for some $C_i > 0,i=1,2$ independent of $t$. On the other hand, it follows from \eqref{est_he} and the equivalence relations \eqref{equi1}-\eqref{equi2} that
\[
\frac{d}{dt} \overline{\mathcal{Y}}^m + \lambda_5\overline{\mathcal{Y}}^m \leq C_1 e^{-C_2t},
\]
for some $\lambda_5 > 0$ which is independent of $t$, where $\overline{\mathcal{Y}}^m$ is appeared in \eqref{def_y2}. 
Applying Gr\"onwall's lemma to the above yields the exponential decay of $\overline{\mathcal{Y}}^m$ towards zero as $t \to \infty$. Note that the $L^2$ norm of $h$ is not included in $\overline{\mathcal{Y}}^m$. For the decay estimate of $\|h(\cdot,t)\|_{L^2}$, we use Lemma \ref{neg_sobo} to get
\[
\|h\|_{L^2} \leq \|h\|_{\dot{H}^{-\frac{d-\alpha}{2}}}^\theta \|\nabla h\|_{L^2}^{1-\theta} \quad \mbox{with} \quad \theta = \frac{1}{1 + \frac{d-\alpha}{2}}.
\]
Since the right hand side of the above converges to zero exponentially fast, we also have the same exponential decay rate of convergence of $\|h\|_{L^2}$. This completes the proof.

\section*{Acknowledgments}
The work of Y.-P. Choi is supported by NRF grant (No. 2017R1C1B2012918) and Yonsei University Research Fund of 2020-22-0505. The work of J. Jung is supported by NRF grant (No. 2019R1A6A1A10073437).

\appendix

\section{Proof of Lemma \ref{h_decay01}}\label{app_a}
In this appendix, we provide the details of proof of Lemma \ref{h_decay01}. 

(i) We apply Proposition \ref{ineqs} to obtain
$$\begin{aligned}
&\lt\|\nabla\lt(\frac{\nabla h}{\rho^2} \cdot \pa^k (\rho u) \rt) \rt\|_{L^2} \cr
&\quad \leq \lt\| \nabla\lt(\frac{\nabla h}{\rho^2} \cdot \lt(\pa^k (h u) -h\pa^k u\rt)\rt) \rt\|_{L^2} + \lt\| \nabla\lt(\frac{h\nabla h}{\rho^2} \cdot \pa^k  u \rt) \rt\|_{L^2}\cr
&\quad \leq C\lt( \lt\| \frac{\nabla h}{\rho^2}\rt\|_{W^{1,\infty}} +\lt\| \frac{h\nabla h}{\rho^2}\rt\|_{W^{1,\infty}}  \rt)\lt( \|\pa^k (hu)-h\pa^k u\|_{H^1} + \|\pa^k u\|_{H^1}\rt)\cr
&\quad \leq C\lt( \frac{\|\nabla^2 h\|_{L^\infty}}{(1-\|h\|_{L^\infty})^2} + \frac{\|\nabla h\|_{L^\infty}^2}{(1-\|h\|_{L^\infty})^3}\rt) \lt( \|\nabla h\|_{W^{1,\infty}}\|\pa^{k}u\|_{H^1} + \|\pa^{k} h\|_{H^1}\|u\|_{L^\infty} + \|\pa^k u\|_{H^1}\rt)\cr
&\quad \leq C\lt(1+\|\nabla h\|_{H^{m-1}} \rt)^2 \|\nabla h\|_{H^{m-1}} \|u\|_{H^{m+\frac{d-\alpha}{2}}},
\end{aligned}$$
where $C=C(m,k,d,\alpha)$ is a positive constant independent of $T$. \newline

(ii) Note that the left hand side equals zero when $k = 1$. For $k \geq 2$, we again use Proposition \ref{ineqs} to estimate
$$\begin{aligned}
&\lt\|\nabla\lt[ \nabla\cdot\lt( \frac1\rho \lt( \pa^k (\rho u) - (\pa^k \rho) u - \rho (\pa^k u)\rt)\rt)\rt] \rt\|_{L^2} \cr
&\quad \leq C\lt\|\nabla^2 \lt(\frac1\rho\rt) \rt\|_{L^\infty} \lt\|  \pa^k (\rho u) - (\pa^k \rho) u - \rho (\pa^k u) \rt\|_{L^2}\cr
&\qquad + C\lt\|\nabla \lt(\frac1\rho\rt) \rt\|_{L^\infty} \lt\| \nabla\lt( \pa^k (\rho u) - (\pa^k \rho) u - \rho (\pa^k u) \rt)\rt\|_{L^2}\cr
&\qquad +\frac{1}{1-\|h\|_{L^\infty}}\lt\| \nabla^2\lt( \pa^k (\rho u) - (\pa^k \rho) u - \rho (\pa^k u) \rt)\rt\|_{L^2}\cr
&\quad\le C\lt(\frac{\|\nabla^2 h\|_{L^\infty}}{(1-\|h\|_{L^\infty})^2} + \frac{\|\nabla h\|_{L^\infty}^2}{(1-\|h\|_{L^\infty})^3} \rt)\lt(\|\nabla h\|_{L^\infty} \|\pa^{k-1} u\|_{L^2} + \|\pa^{k-1} h\|_{L^2} \|\nabla u\|_{L^\infty} \rt) \\
&\qquad + C\frac{\|\nabla h\|_{L^\infty}}{(1-\|h\|_{L^\infty})^2}\Big[\lt(\|\nabla^2h\|_{L^\infty}  \|\pa^{k-1}u\|_{L^2} + \|\nabla \pa^{k-1} h\|_{L^2} \|\nabla u\|_{L^\infty}\rt)\\
&\hspace{4cm} + \lt(\|\nabla h\|_{L^\infty} \|\pa^{k-1}\nabla u\|_{L^2}  + \|\pa^{k-1} h\|_{L^2} \|\nabla^2 u\|_{L^\infty}\rt)\Big] \\
&\qquad + \frac{C}{1-\|h\|_{L^\infty}}\Big[\lt(\|\nabla^3 h\|_{L^\infty}\|\pa^{k-1}u\|_{L^2} + \|\pa^{k-1}\nabla^2 h\|_{L^2} \|\nabla u\|_{L^\infty} \rt)\\
&\hspace{4cm} + \lt(\|\nabla^2 h\|_{L^\infty} \|\pa^{k-1} \nabla u\|_{L^2} + \|\pa^{k-1} \nabla h\|_{L^2} \|\nabla^2 u\|_{L^\infty} \rt)\\
&\hspace{5cm} + \lt(\|\nabla h\|_{L^\infty} \|\pa^{k-1}\nabla^2u\|_{L^2} + \|\pa^{k-1} h\|_{L^2}\|\nabla^3 u\|_{L^\infty} \rt) \Big] \\
&\quad\le C\lt( 1+ \|\nabla h\|_{W^{1,\infty}}\rt)^2\|\nabla h\|_{H^{m-1}}\|u\|_{H^{m+\frac{d-\alpha}{2}}}\\
&\quad\le C\lt( 1+ \|\nabla h\|_{H^{m-1}}\rt)^2\|\nabla h\|_{H^{m-1}}\|u\|_{H^{m+\frac{d-\alpha}{2}}}.
\end{aligned}$$

(iii) Note that
\[
 \pa\lt(\frac{1}{\rho^2} (\pa^k h) \nabla h\rt) = \frac{1}{\rho^2}\lt(-\frac{\pa h}{2\rho}(\pa^k h) \nabla h + \pa^{k+1}h \nabla h + (\pa^k h) \nabla \pa h \rt),
\]
and thus taking $L^2$ norm on the both sides of above and using Proposition \ref{ineqs} gives
$$\begin{aligned}
&\lt\| \pa\lt(\frac{1}{\rho^2} (\pa^k h) \nabla h\rt)\rt\|_{L^2} \cr
&\quad \leq \frac{C}{(1-\|h\|_{L^\infty})^2}\lt(\frac{\|\nabla h\|_{L^\infty}^2}{2(1-\|h\|_{L^\infty})} \|\pa^k h\|_{L^2} + \|\pa^{k+1} h\|_{L^2}\|\nabla h\|_{L^\infty} + \|\pa^k h\|_{L^2} \|\nabla^2 h\|_{L^\infty} \rt)\cr
&\quad \leq C\lt(1+\|\nabla h\|_{H^{m-1}}\rt)\|\nabla h\|_{H^{m-1}}^2.
\end{aligned}$$

(iv) It follows from the equation of $u$ in \eqref{np_ER2} that
$$\begin{aligned}
-\Lambda^{d-\alpha}\pa^{k-1} \pa_t u &= \Lambda^{d-\alpha} \pa^{k-1} \lt( u \cdot \nabla u + u + \Lambda^{\alpha-d} \nabla h\rt)\cr
&= u \cdot \nabla \Lambda^{d-\alpha} \pa^{k-1} u + [\Lambda^{d-\alpha}, u \cdot \nabla]\pa^{k-1} u \cr
&\quad + \Lambda^{d-\alpha} [\pa^{k-1}, u \cdot \nabla ]u +\Lambda^{d-\alpha}\pa^{k-1}u + \pa^{k-1} \nabla h.
\end{aligned}$$
Thus we have
$$\begin{aligned}
&\lt\|\Lambda^{d-\alpha}\pa^{k-1} \pa_t u\rt\|_{L^2}  \cr
&\quad \leq C\bigg( \|u\|_{L^\infty} \|\nabla \Lambda^{d-\alpha} \pa^{k-1} u\|_{L^2}  + \|u\|_{H^{\frac{d}{2}+1+(d-\alpha)+\e}} \|\pa^{k-1} u\|_{H^{d-\alpha}} \cr
&\hspace{4cm} + \sum_{j=1}^{k-1} \|\Lambda^{d-\alpha} (\pa^j u \cdot \nabla \pa^{k-1-j} u)\|_{L^2} + \|\Lambda^{d-\alpha} \pa^{k-1} u\|_{L^2} + \|\pa^{k-1} \nabla h\|_{L^2}\bigg)\\
&\quad \le C \lt(\|u\|_{H^{m+\frac{d-\alpha}{2}}}^2 + \|u\|_{H^{m+\frac{d-\alpha}{2}}} + \|\nabla h\|_{H^{m-1}}\rt),
\end{aligned}$$
where we used Lemma \ref{comm_est2} with $\e>0$ satisfying $(d-\alpha)/2 + \e <1$, used Proposition \ref{ineqs} to get
\[\begin{aligned}
\|\Lambda^{d-\alpha} ( \pa^j u \cdot \nabla \pa^{k-1-j} u)\|_{L^2} &\lesssim_{\alpha,d} \| \pa^j u \cdot \nabla \pa^{k-1-j} u\|_{H^2}\\
&\lesssim_{\alpha,d,k,j} \|\pa^j u\|_{H^{m-j}} \|\pa^{k-j} u\|_{H^{m-(k-j)}}  \\
&\lesssim_{\alpha, d, k,j} \|u\|_{H^m}^2,
\end{aligned}\]
for $1 \le j \le k-1$. \newline

(v) By adding and subtracting, we obtain
$$\begin{aligned}
&\Lambda^{d-\alpha} \pa^{k-1}  (\nabla u : (\nabla  u)^T) \cr
&\quad = \sum_{i=1}^d  \Lambda^{d-\alpha} \pa^{k-1}  (\pa_{x_i} u \cdot \nabla  u_i)\cr
&\quad =\sum_{i=1}^d \Lambda^{d-\alpha} \lt((\pa^{k-1}  \pa_{x_i} u) \cdot \nabla  u_i  +\pa_{x_i} u \cdot \nabla  \pa^{k-1} u_i \rt) \cr
&\qquad + \sum_{i=1}^d \Lambda^{d-\alpha} \lt( \pa^{k-1}(\pa_{x_i} u \cdot \nabla u_i) - \pa_{x_i} u \cdot \nabla \pa^{k-1} u_i  - \pa_{x_i} (\pa^{k-1} u) \cdot \nabla u_i\rt)\cr
&\quad =2\sum_{i=1}^d \pa_{x_i} u \cdot \nabla \Lambda^{d-\alpha} \pa^{k-1}  u_i +2[\Lambda^{d-\alpha}, \pa_{x_i} u \cdot \nabla] \pa^{k-1}  u_i\cr
&\qquad + \sum_{i=1}^d \Lambda^{d-\alpha} \lt( \pa^{k-1}(\pa_{x_i} u \cdot \nabla u_i) - \pa_{x_i} u \cdot \nabla \pa^{k-1} u_i  - \pa_{x_i} (\pa^{k-1} u) \cdot \nabla u_i\rt).
\end{aligned}$$
We then apply Lemma \ref{comm_est2} and Proposition \ref{ineqs} to deduce
$$\begin{aligned}
&\|\Lambda^{d-\alpha} \pa^{k-1}  (\nabla u : (\nabla  u)^T)\|_{L^2}\cr
&\quad \leq C\bigg( \|\nabla u\|_{L^\infty} \|u\|_{H^{k+(d-\alpha)}} + \|\nabla u\|_{H^{\frac{d}{2}+1+(d-\alpha)+\e}} \|\pa^{k-1} u\|_{H^{d-\alpha}} + \sum_{\ell=1}^{k-2} \|\Lambda^{d-\alpha}(\pa^{\ell+1} u \cdot \nabla \pa^{k-1-\ell}u)\|_{L^2} \bigg)\cr
&\quad \leq C\|u\|_{H^{m+\frac{d-\alpha}{2}}}^2,
\end{aligned}$$
where $\e$ satisfies $(d-\alpha)/2+\e <1$.
%
%
%
%

\end{document}